\documentclass[a4paper,12pt]{amsart}
\usepackage[left=2.5cm, right=2.5cm, top=3cm, bottom=3cm]{geometry}
\usepackage[colorlinks=true,linkcolor=red,citecolor=blue]{hyperref}

\usepackage{amsmath,amscd,amsthm,amssymb,mathtools}
\usepackage{color}
\usepackage[utf8]{inputenc}
\normalfont
\usepackage[T1]{fontenc}
\usepackage{tikz-cd}
\usepackage{tikz}
\usepackage{enumerate}
\usepackage{comment}
\usepackage{mathabx}
 \usepackage[all]{xy}

\usetikzlibrary{decorations.pathmorphing,shapes}
\usetikzlibrary{arrows.meta}
\tikzcdset{arrow style=tikz,
    squigarrow/.style={
        decoration={
        snake, 
        amplitude=.4mm,
        segment length=2mm
        }, 
        rounded corners=.2pt,
        decorate
        }
    }

\newtheorem{theo}{Theorem}[section]
\newtheorem{lemm}[theo]{Lemma}
\newtheorem{prop}[theo]{Proposition}
\newtheorem{coro}[theo]{Corollary}

\theoremstyle{definition}
\newtheorem{defi}[theo]{Definition}
\newtheorem{rema}[theo]{Remark}

\newtheorem{exam}[theo]{Example}

\theoremstyle{theorem}

\newcommand{\CC}{\mathbb{C}}

\newcommand{\NN}{\mathbb{N}}

\newcommand{\PP}{\mathbb{P}}
\newcommand{\QQ}{\mathbb{Q}}
\newcommand{\RR}{\mathbb{R}}

\newcommand{\ZZ}{\mathbb{Z}}
\newcommand{\Aa}{\mathcal{A}}

\newcommand{\Cc}{\mathcal{C}}
\newcommand{\Dd}{\mathcal{D}}

\newcommand{\Mm}{\mathcal{M}}

\newcommand{\Oo}{\mathcal{O}}
\newcommand{\Pp}{\mathcal{P}}

\newcommand{\Zz}{\mathcal{Z}}

\newcommand{\Ho}{\mathrm{Ho}}

\newcommand{\Img}{\mathrm{Im}}

\newcommand{\ch}[1]{\mathrm{Ch^*({#1})}}

\newcommand{\MHS}{{\mathsf{MHS}}}
\newcommand{\MHC}{{\mathsf{MHC}}}
\newcommand{\MHD}{{\mathsf{MHD}}}

\newcommand{\Hom}{\mathrm{Hom}}

\title[Products in absolute Hodge cohomology and mixed Hodge formality]{Products in absolute Hodge cohomology and mixed Hodge formality}

\author{Pedro Magalhães}

\address[P. Magalhães]{Departament de Matemàtiques i Informàtica, Universitat de Barcelona\\
Gran Via 585}
\email{pedrommagalhaes1998@gmail.com}

\thanks{This work was supported by FCT - Fundação para a Ciência e Tecnologia, I.P. by project 2021.06151.BD with DOI identifier https://doi.org/10.54499/2021.06151.BD.
Partial financial support from the Spanish State Research Agency through projects PID2020-117971GB-C22, PID2024-155646NB-I00 and EUR2023-143450.}

\begin{document}

\maketitle 

\begin{abstract}
  Given a complex algebraic variety, we relate the product structure on its absolute Hodge cohomology with mixed Hodge formality, a notion of formality in the sense of rational homotopy, that takes into account the mixed Hodge structures present in complex algebraic geometry. Namely, we prove that mixed Hodge formality implies the splitting of absolute Hodge cohomology as a square-zero extension. We define and study certain products, akin to Massey products, which live in the absolute Hodge cohomology and obstruct its splitting. We also relate mixed Hodge formality to the existence of multiplicative Chow-Künneth decompositions.
\end{abstract}

\section{Introduction}
In rational homotopy, a fundamental construction is Sullivan's functor of piecewise linear forms. This sends any topological space $X$ to $\Aa_{pl}(X)$, a commutative differential graded algebra (cdga for short) defined over the rationals. It contains all the rational homotopy information of the space and its cohomology is just singular cohomology with rational coefficients. The algebra $\Aa_{pl}(X)$ may be thought of as a replacement of de Rham's cdga defined over $\RR$ for smooth manifolds. In fact, when $X$ is such a smooth manifold, these cdga's are quasi-isomorphic $\Aa_{pl}(X)\otimes\RR\simeq \Aa_{\mathrm{dR}}(X)$. When $X$ is a complex algebraic variety, then $\Aa_{pl}(X)$ is endowed with extra structure which contains complex geometric information about the variety \cite{DeHII}, \cite{DeHIII}, \cite{morgan}, \cite{navarro}. More precisely, there is a functor
\[ \mathsf{Var}_\CC \xrightarrow{\Aa} \Ho(\MHD), \]
where $\MHD$ denotes the category of \textit{mixed Hodge diagrams}. Its objects are partly given by a rational cdga with a weight filtration and a complex cdga with a Hodge filtration, which are quasi-isomorphic and induce \textit{mixed Hodge structures} on cohomology. Here $\Ho$ denotes the localization with respect to quasi-isomorphisms of mixed Hodge diagrams. In particular, there is a forgetful functor 
\[\MHD \to \mathsf{cdga}_\QQ, \]
which maps $\Aa(X)$ to $\Aa_{pl}(X)$. The cohomology of a mixed Hodge diagram is again a mixed Hodge diagram, and an interesting question to ask is whether 
\[\Aa(X) \cong H^*(X) \]
in $\Ho(\MHD)$. In this case, $X$ is said to be \textit{mixed Hodge formal}. If $X$ is a smooth projective variety, then $H^*(X)$ has pure Hodge structures, and thus $\Aa_{pl}(X)$ is formal \cite{DGMS}. However, not every smooth projective variety is mixed Hodge formal \cite{CCM}. By \cite{MHF}, there are successively defined obstructions to mixed Hodge formality
\[ \theta_k \in \mathrm{Harr}_{\mathrm{DB}}^{k}(H^*(X)) \qquad \text{ for } k \geq 3, \]
where $\mathrm{Harr}_{\mathrm{DB}}^{k}(H^*(X))$ is a variant of Harrison cohomology related to the notion of \textit{absolute Hodge cohomology}, introduced by Beilinson \cite{Beilinson} (see also \cite{DBEsnVie} for a basic reference on the subject). In this work, we will refer to absolute Hodge cohomology as \textit{Deligne-Beilinson cohomology}. In the literature, this terminology is sometimes used for a weaker notion where one forgets the weight filtration. Given a complex algebraic variety $X$, the Deligne Beilinson cohomology of $X$
\[ H_{\mathrm{DB}}^*(X) = \bigoplus_{p \geq 0} H_{\mathrm{DB}}^*(X,p) \]
is a bigraded group that refines singular cohomology. Indeed, for $p = 0$, the group $H_{\mathrm{DB}}^*(X,0)$ is isomorphic to $H^*(X;\ZZ)$ and for higher values of $p$, $H_{\mathrm{DB}}^*(X,p)$ captures extra complex geometric information of $X$. For instance, when $X$ is smooth and projective, Deligne-Beilinson cohomology reduces to the notion of Deligne cohomology and, in particular, for $p=1$, we have
\[ H_{\mathrm{DB}}^n(X,1) \cong H^{n-1}(X,\Oo^*_X), \]
where $\Oo^*_X$ is the sheaf of invertible holomorphic functions on $X$. In particular, $H^2_{\mathrm{DB}}(X,1)$ is the Picard group of $X$. Moreover, $H^2_{\mathrm{DB}}(X,2)$ is isomorphic to the group of isomorphism classes of line bundles with a connection on $X$ \cite{Beilinson2}. For higher values of $p$, there are analogous classification results \cite{DcohoBryl} \cite{GeoDcohoGajer}. Here, we work with coefficients in a subfield $\Bbbk \subset \RR$, instead of $\ZZ$. An important property of Deligne-Beilinson cohomology is that it fits in a short exact sequence of $\Bbbk$-vector spaces:
\[ 0 \to J^p H^{2p-1+*}(X) \to H_{\mathrm{DB}}^*(X,p) \to \mathrm{Hdg}^{2p+*}(X,p) \to 0, \]
where 
\begin{align*}
  &\mathrm{Hdg}^*(X,p) := (2 \pi i)^p W_{2p} H^*(X;\Bbbk) \cap F^p W_{2p} H^*(X;\CC), \\
  &J^p H^*(X) = \frac{W_{2p}H^*(X;\CC)}{W_{2p} H^*(X;\Bbbk) + F^p W_{2p} H^*(X;\CC)}.
\end{align*}
Here $W$ and $F$ denote the weight and Hodge filtrations of the mixed Hodge structure on $H^*(X)$. Let us also denote
\begin{align*}
  &\mathrm{Hdg}^*(X) := \bigoplus_p \mathrm{Hdg}^*(X,p)[2p]\\
  &J^p H^*(X) = \bigoplus_p J^p H^*(X)[2p-1].
\end{align*}
In particular, when $X$ is smooth and projective, $\mathrm{Hdg}^0(X)$ is the space of Hodge classes of $X$ and $\mathrm{Hdg}^{<0}(X) = 0$. The product of $H^*(X)$ makes $\mathrm{Hdg}^*(X)$ into a commutative algebra and $JH^*(X)$ a $\mathrm{Hdg}^*(X)$-module. Furthermore, there is an associative and graded commutative product on $H^*_{\mathrm{DB}}(X)$ which makes it a square-zero extension
\begin{equation}
  \label{sqzeroext-intro-eq}
  0 \to JH^*(X) \to H_{\mathrm{DB}}^*(X) \to \mathrm{Hdg}^*(X) \to 0. 
\end{equation}
Let now $X$ be a smooth complex projective variety. The homotopy category of mixed Hodge complexes is equivalent to $D(\MHS)$ the derived $1$-category of mixed Hodge structures, as symmetric monoidal categories. The mixed Hodge diagram $\Aa(X)$ is, thus, in particular a commutative algebra in $D(\MHS)$. Our first main result gives a characterization of the first obstruction $\theta_3$ to mixed Hodge formality. We prove:
\begin{theo}
  The class $\theta_3$ is the only obstruction to the existence of an isomorphism $\Aa(X) \cong H^*(X)$, seen as algebras in the derived category of mixed Hodge structures $D(\MHS)$ together with the derived tensor product.
\end{theo}
Note that this is a considerably weaker condition than mixed Hodge formality, since the category of algebras in $D(\MHS)$ does not see the higher coherence conditions present in the homotopy category of $\MHD$. As a consequence, the vanishing of the first obstruction $\theta_3$ implies that the extension (\ref{sqzeroext-intro-eq}) is split. We also use this characterization of $\theta_3$ to show that that the existence of a \textit{multiplicative Chow-Künneth decomposition} on $X$ implies the vanishing of $\theta_3$.

We then define Massey-like products in Deligne-Beilinson cohomology which obstruct the splitting of $H_\mathrm{DB}(X)$ as a square-zero extension. Given classes $[a], [b] \in \mathrm{Hdg}^*(X)$ such that $[a][b] = 0$, we define the \textit{double Deligne-Beilinson Massey product} of $[a]$ and $[b]$:
\[ \langle [a], [b]  \rangle_{\mathrm{DB}} \in H^*_{\mathrm{DB}}(X). \]
These products are defined for general complex algebraic varieties and obstruct mixed Hodge formality. More concretely, if $H^*_{\mathrm{DB}}(X)$ is a split square-zero extension, then all double Deligne-Beilinson Massey products vanish. If $X$ is smooth and projective and $[a]$ and $[b]$ are the cohomology classes associated to algebraic cycles $A$ and $B$, we show that 
\[ \langle [a], [b]  \rangle_{\mathrm{DB}}  = \mathcal{AJ}(A \cdot B), \]
where $\mathcal{AJ}$ is the Abel-Jacobi map. 

We also review the notion of triple DB Massey products, which are triple Massey products in a certain cdga whose cohomology is $H_\mathrm{DB}(X)$, and prove that they vanish when the second obstruction $\theta_4$ vanishes. Unlike the double DB Massey products, these higher products - and variants of them - have been studied before. In \cite{Den}, Deninger defines triple Massey products in real Deligne cohomology and computes non-trivial ones on an elliptic curve. In \cite{Wen}, Wenger generalizes these products to absolute cohomology theories. Then, in \cite{MPDBcoho-thesis}, Schwarzhaupt studies such products for cohomology classes arising from algebraic cycles, in the setting of integer-valued Deligne-Beilinson cohomology of smooth complex projective varieties. Products of this type are also studied in \cite{polysymb}. In \cite{MPDifcohostack}, Massey products are further generalised to differential cohomology theories, using the language of higher stacks. In this work, instead of constructing such triple Massey products directly for algebraic varieties (as in the articles previously mentioned), we define them for mixed Hodge diagrams. We also define triple DB Massey products only for Hodge classes and not general classes in Deligne-Beilinson cohomology, as these products for Hodge classes are more naturally related to mixed Hodge formality.

In complex geometry, there is another notion triple Massey products, introduced by Angella and Tomassini \cite{DaniToma} for complex manifolds, called triple ABC Massey products. These products obstruct strong formality, a natural notion of formality in the setting of pluripotential homotopy theory \cite{Stel-pluri}. There are many new results on the (non)-vanishing of triple ABC Massey products (see \cite{SfTo}, \cite{Steletall-nonformal}, \cite{StelGioZoller}, \cite{StelMar-Mer}). As observed in \cite{Stel-pluri}, for compact Kähler manifolds, strong formality implies mixed Hodge formality (for mixed Hodge structures with coefficients in $\RR$). One should thus expect there to be a relation between triple ABC Massey products and triple DB Massey products, in this setting. As communicated to us by Jonas Stelzig and Dan Petersen, such products coincide. This is also a consequence of the fact the second obstruction $\theta_4$ evaluates to both these types of products, as mentioned in Remark \ref{ABC=DB-rema}.
\medskip

\noindent
\textbf{\textbf{Organization of the paper:}} We start by recalling basic definitions of mixed Hodge structures and mixed Hodge complexes. In section \ref{MHD-sec}, we review the notions of mixed Hodge diagrams, their homotopy models and the theory of obstructions to mixed Hodge formality introduced in \cite{MHF}. In section \ref{DBcoho-sec}, we review the notion of Deligne Beilinson cohomology for general mixed Hodge complexes and the complexes computing it $C_\mathrm{DB}(-)$ and $P_\mathrm{DB}(-)$. Here, we compute the homotopy transfer of the commutative product on the latter to the former. In section \ref{SplitDBcoho-sec}, we prove Theorem \ref{phi2-theo}, giving a characterization of the first obstruction $\theta_3$ to mixed Hodge formality. We then define, in section \ref{DBprods-sec}, double and triple products in Deligne-Beilinson cohomology and relate them to mixed Hodge formality. In the final section, we get back to complex geometry and compute double Deligne-Beilinson products of algebraic cycles in a smooth complex projective variety $X$, in terms of the Abel-Jacobi map. We finish by relating the existence of a multiplicative Chow-Künneth decomposition on $X$ with mixed Hodge formality of $X$.

\medskip

\noindent
\textbf{\textbf{Acknowledgments:}} The author would like to thank Joana Cirici for introducing him to the ideas around mixed Hodge formality and for reviewing this text. The author would also like to thank Jonas Stelzig for all the helpful discussions and for pointing him to the notion of absolute Hodge cohomology. Finally, the author would like to thank Dan Petersen, who pointed him to the relation between multiplicative Chow-Künneth decompositions and mixed Hodge formality. 

\section{Mixed Hodge complexes}
We start by reviewing the concepts of mixed Hodge structures and mixed Hodge complexes. We refer to \cite{PS-mhs} for a detailed exposition of the subject.
\subsection{Mixed Hodge structures}
Let $\Bbbk$ be a subfield of $\RR$.
\begin{defi}
  A $\Bbbk$-\textit{pure Hodge structure} of weight $n \in \ZZ$ is a pair $(V,F)$, where $V$ is a finite-dimensional $\Bbbk$-module and $F$ is a descending filtration on $V_\CC = V \otimes \CC$ such that
  \[ V_\CC \cong F^p V_\CC \oplus \overline{F^{n-p+1}V_\CC},\]
  for $p \in \ZZ$, where $(\overline{-})$ denotes complex conjugation.
\end{defi}

\begin{defi}
  A $\Bbbk$-\textit{mixed Hodge structure} is a triple $(V,W,F)$ consisting of a finite dimensional $\Bbbk$-module $V$, an increasing filtration $W$ on $V$, called the
  \textit{weight filtration} and a decreasing filtration $F$ on $V_\CC$, called
  the \textit{Hodge filtration}, such that on
  \begin{equation*}
    Gr_n^W(V) = W_n V / W_{n-1} V,
  \end{equation*}
  the filtration $F$ induces a pure Hodge structure of weight $n$, for any $n \in \ZZ$.
\end{defi}
A morphism of $\Bbbk$-mixed Hodge structures is a morphism of $\Bbbk$-modules that preserves both filtrations. In the following, we may drop the $\Bbbk$- from the notation when the field is irrelevant for the discussion.
\begin{rema}
    Note that when the weight filtration is of the form 
    \[ W_{n-1} = 0 \subset W_n = V, \]
    for some $n \in \ZZ$, then the mixed Hodge structure $(V,W,F)$ is just a pure Hodge structure of weight $n$.
\end{rema}
The category $\MHS_\Bbbk$ of $\Bbbk$-mixed Hodge structures is abelian closed symmetric monoidal with the tensor product and inner hom of $\Bbbk$-modules together with the induced filtrations:
\begin{align}
  \label{homfil-eq}
  &W_n(A \otimes_\Bbbk B) = \bigoplus_{i+j = n} W_i A \otimes_\Bbbk W_j B  \nonumber \\
  &W_n \, \underline{\Hom}_\Bbbk(A,B) = \{ f \in \underline{\Hom}_\Bbbk(A,B) | f(W_i A) \subset W_{i+n}B \}.
\end{align} 
The category $\MHS_\Bbbk$ has non-trivial extensions:
\begin{prop}[\cite{mhsext-carlson}, see also \cite{PS-mhs}]
  \label{mh-ext}
  Given $A,B \in \MHS_\Bbbk$, 
  \begin{equation*}
    \mathrm{Ext}_{\MHS_\Bbbk}^1(A,B) \cong\frac{W_0\Hom_{\CC}(A_\CC,B_\CC)}{W_0 \Hom_{\Bbbk}^W(A,B) + W_0F^0\Hom_\CC(A_\CC,B_\CC)}.
  \end{equation*}
  Furthermore, $\mathrm{Ext}_{\MHS_\Bbbk}^n(A,B) = 0$ for $n \geq 2$. \vspace{1ex} \\
\end{prop}
\begin{defi}
    A mixed Hodge structure is said to be \textit{split} if it is isomorphic to a direct sum of pure Hodge structures of possibly different weights.
\end{defi}
\begin{rema}
    \label{delsplit-rema}
    Although many mixed Hodge structures are not split, for any $(V,W,F)$ there is a splitting of $V_\CC \cong \oplus I^{p,q}$, called the \textit{Deligne splitting}. It satisfies
    \[ \overline{I^{p,q}} \equiv I^{q,p} \quad (\mathrm{mod} \; W_{p+q-2}). \]
    Hence, if the weight filtration is of the form 
    \[ W_{n-2} = 0 \subset W_{n-1} \subset W_n = V, \]
    that is, it has lenght $\leq 1$, then the $\RR$-mixed Hodge structure $(V \otimes \RR,W,F)$ is split.
\end{rema}
\subsection{Mixed Hodge complexes}
In the following, by a filtered complex we mean a complex in the category of filtered $\Bbbk$-modules where the filtrations are assumed to be exhaustive and Hausdorff. That is:
\[ \bigcup_{k} W_k A = A, \qquad \bigcap_k W_k A = 0, \qquad \bigcap_k W_k H^*(A) = 0,\]
where $H^*(A)$ has the usual induced filtration:
\[ W_k H^*(A) = \Img( H^*(W_k A) \to H^*(A)). \]
A filtered complex $(A,W)$ is said to be \textit{regular} if for every degree $n \in \ZZ$, there exists $k \in \ZZ$ such that 
\[ W_k A^n = 0. \]
A filtered complex $(A,W)$ is said to be \textit{biregular} if for every degree $n \in \ZZ$, there exist $p,q \in \ZZ$ such that
\[ W_p A = 0, \qquad W_q A = A. \]
A filtered complex $(A,W)$ is said to be \textit{$d$-strict} if 
\[ d(W_k A) = W_k A \cap \Img(d). \]
If $(A,W)$ is regular, this is equivalent to the associated spectral sequence degenerating at $E_1$. A morphism of filtered complexes $f :(A,W) \to (B,W)$ (or a filtered morphism) is said to be a \textit{filtered quasi-isomorphism} if for every $k \in \ZZ$, the induced graded
\[ Gr_k^W(f) : Gr^W_k A \to Gr^W_k B \]
is a quasi-isomorphism.
\begin{rema}
  In the literature, the notion of filtered quasi-isomorphism usually refers to a filtered morphism $f : (A,W) \to (B,W)$ such that for every $k \in \ZZ$,
  \[ W_k f : W_k A \to W_k B \]
  is a quasi-isomorphism. If $(A,W)$ and $(B,W)$ are regular, then the two notions are equivalent.
\end{rema}
\begin{defi}
  \label{mhc-defi}
  A $\Bbbk$-\textit{mixed Hodge complex of length $s$} is given by a filtered cochain complex $(A_\Bbbk,W)$ over $\Bbbk$, a bifiltered cochain complex $(A_\CC,W,F)$ over $\CC$, filtered complexes $(A_i,W)$ over $\CC$ for $1 \leq i \leq s-1$ and filtered quasi-isomorphisms $\varphi_{u}$ for each arrow $u : i \to j$ as in the following diagram:
  \[ \begin{tikzcd}[column sep = small]
    & (A_1,W) & & \cdots & \\
    (A_0,W) = (A_\Bbbk,W) \otimes \CC \arrow[ru,"\varphi_{01}"] & & (A_2,W) \arrow[lu,"\varphi_{21}"'] \arrow[ru,"\varphi_{23}"] & & \arrow[lu,"\varphi_{s s-1}"'] (A_s,W) = (A_\CC,W)  
  \end{tikzcd} \]
  In addition, the following axioms are satisfied:
  \begin{enumerate}
    \item The weight filtrations $W$ are increasing, exhaustive and regular. The Hodge filtration $F$
    is decreasing and biregular. The cohomology $H^*(A_\Bbbk)$ has finite type.
    \item \label{mhc-axiom2} The filtered complexes $(A_\CC,W)$, $(A_\CC,F)$, $(Gr_k^W(A_\CC),F)$ and $(Gr^p_F(A_\CC),W)$ for all $k,p \in \ZZ$ are $d$-strict.
    \item \label{mhc-axiom3} For all $n \in \NN$ and $p \in \ZZ$, the filtration induced by $F$ on $H^n(Gr_p^W A_\CC)$ and the isomorphisms $\varphi_{u}^*$ induce a pure Hodge structure of weight $p$ on $H^n(Gr_p^W A_\Bbbk)$.
  \end{enumerate}
\end{defi}
The morphisms $\varphi_u : (A_\Bbbk,W) \otimes \CC \to (A_\CC,W)$ are called the \textit{comparison morphisms}. A morphism of mixed Hodge complexes is a tuple of morphisms
\[ f_\Bbbk : (A_\Bbbk, W) \to (B_\Bbbk,W), \quad f_\CC : (A_\CC, W,F) \to (B_\CC,W,F), \]
\[ f_i : (A_i,W) \to (B_i,W), \quad \text{for all } 0 \leq i \leq s, \]
that commute with the comparison morphisms and such that $f_0 = f_\Bbbk \otimes \CC$ and $f_s = f_\CC$. The category $\MHC_\Bbbk$ of $\Bbbk$-mixed Hodge complexes is symmetric monoidal with tensor prodcut defined component-wise. To simplify notation, we will drop the length $s$ from the notation, when the results apply to any length.
\begin{rema}
    Given a mixed Hodge complex $A$, its cohomology 
    \[ (H^*(A_\Bbbk),H^*(A_\CC),\varphi^*) \] 
    with the induced filtrations, is again a mixed Hodge complex. The axioms \ref{mhc-axiom2} and \ref{mhc-axiom3} imply that the triple $(H^*(A_\Bbbk),W,(\phi^*)^{-1}F)$ is a graded $\Bbbk$-mixed Hodge structure.
\end{rema}
A morphism $f : A \to B$ of mixed Hodge complexes is said to be a \textit{quasi-isomorphism} if $f^* : H^*(A) \to H^*(B)$ is an isomorphism.
\vspace{1ex}
\par There is an inclusion $\mathrm{Ch}^*(\MHS) \hookrightarrow \MHC$ given by setting the comparison morphisms to be identities. This induces an equivalence of categories 
\[ D(\MHS) \xrightarrow{\sim} \Ho(\MHC) := \MHC[\mathrm{Qiso}^{-1}], \]
(see \cite{Beilinson}).
\begin{rema}
  \label{del-vs-bei}
  Definition \ref{mhc-defi} does not exactly coincide with Deligne's definition of a mixed Hodge complex (cf. \cite{DeHIII}). In Deligne's definition, axioms (\ref{mhc-axiom2}) and (\ref{mhc-axiom3}) are replaced with
  \begin{itemize}
    \item[(2')] For each $k \in \ZZ$, the filtered complex $(Gr_k^W(A_\CC),F)$ is $d$-strict,
    \item[(3')] For all $n \in \NN$ and $p \in \ZZ$, the filtration induced by $F$ on $H^{n}(Gr_p^W A_\CC)$ and the maps $\varphi_u^*$ induce a pure Hodge structure of weight $p+n$ on $H^n(Gr_p^W A_\Bbbk)$.
  \end{itemize}
  Instead, our definition coincides with Beilinson's notion of absolute Hodge complex \cite{Beilinson}. There is, however, a functor from one category to the other. Denote by $\MHC'$ the category of mixed Hodge complexes as defined by Deligne. Then, there exist the Decalage functor
  \[ \mathrm{Dec} : \mathsf{MHC}' \to \MHC, \]
  defined by changing the weight filtration of each $A_i$ in the following way
  \[ \mathrm{Dec}(W)_kA^n = W_{k-n} A^n \cap d^{-1}(W_{k-n-1}A^{n+1}). \]
  This functor becomes an equivalence of categories after localizing at quasi-isomorphisms (see section $4$ of \cite{CG2}).
\end{rema}

\section{Mixed Hodge diagrams}
\label{MHD-sec}
In this section, we review the notions of mixed Hodge diagrams, their homotopy models and, in particular, the homotopy transfer theorem in the setting of mixed Hodge structures, introduced in \cite{MHF}.
\begin{defi}
  The category $\mathsf{MHD}$ of \textit{mixed Hodge diagrams} is the category of commutative monoids in $\MHC$.
\end{defi}
\begin{rema}
  A mixed Hodge diagram $A$ is equivalently defined as a mixed Hodge complex with the extra structure of a filtered cdga on $(A_\Bbbk,W)$, a bifiltered cdga on $(A_\CC,W,F)$  and filtered cdga's on $(A_i,W)$ for all $0 \leq i \leq s$, such that the comparison morphism $\phi_u : (A_i,W) \to (A_j,W)$  is a filtered quasi-isomorphism of algebras for every arrow $u : i \to j$.
\end{rema}
We say that a morphism of mixed Hodge diagrams is a \textit{quasi-isomorphism} if its underlying morphism of mixed Hodge complexes is a quasi-isomorphism. Denote by $\Ho(\MHD)$ the category of mixed Hodge diagrams localized at quasi-isomorphisms. The homotopy theory of mixed Hodge diagrams is studied in \cite{Cir-cofmodel}, where Cirici constructs cofibrant and minimal models of mixed Hodge diagrams. There is also an $\infty$-algebra variant of mixed Hodge diagrams which captures their homotopy theory.
\begin{defi}
  \label{Pinftymodel-defi}
  A \textit{$C_\infty$-mixed Hodge diagram} is a mixed Hodge complex 
  \[A = ((A_\Bbbk,W),(A_\CC,W,F),\phi) \] 
  together with
  \begin{itemize}
    \item a $C_\infty$-algebra structure $m^\Bbbk$ on $A_\Bbbk$ such that the maps $(m^\Bbbk)_n$ preserve $W$ for all $n \geq 2$,
    \item a $C_\infty$-algebra structure $m^\CC$ on $A_\CC$ such that the maps $(m^\CC)_n$ preserve both $W$ and $F$ for all $n \geq 2$,
    \item for every $0 \leq i \leq s$, a $C_\infty$-algebra structure $m^i$ on $A_i$ such that the maps $(m^i)_n$ preserve $W$ for all $n \geq 2$,
    \item for every $u : i \to j$, an $\infty$-morphism 
    \[ \hat{\phi}^u: (A_i,m^i) \to (A_j,m^j) \]
    such that $\hat{\phi}^u_n$ preserves $W$ for all $n \geq 1$ and $\hat{\phi}^u_1 = \phi_u$.  
  \end{itemize}
\end{defi}
\begin{defi}
  An \textit{$\infty$-morphism} of $C_\infty$-mixed Hodge diagrams $f : A \to B$ is a collection of $\infty$-morphisms
\[ f^\Bbbk : (A_\Bbbk,m^\Bbbk) \to (B_\Bbbk,m^\Bbbk), \quad f^\CC : (A_\CC,m^\CC) \to (B_\CC,m^\CC) \]
\[ f^i : (A_i,m^i) \to (B_i,m^i), \quad \text{for all } 0 \leq i \leq s, \]
such that $f^\Bbbk_n$ and $f^i_n$ preserve $W$ and $f^\CC_n$ preserves $W$ and $F$ for all $n \geq 1$. Moreover, for every $u : i \to j$, they make the following square of $\infty$-morphisms commute:
\[ \begin{tikzcd}
      A_i \arrow[r,"f^i"] \arrow[d, "\hat{\phi}^u_A"] & B_i \arrow[d, "\hat{\phi}^u_B"] \\
      A_j \arrow[r, "f^j"] & B_j.
    \end{tikzcd} 
\]
\end{defi}
We say that an $\infty$-morphism $f = (f^\Bbbk,f^i,f^\CC)$ is an \textit{$\infty$-quasi-isomorphism} if the underlying morphism of mixed Hodge complexes $f_1 = ((f^\Bbbk)_1,(f^i)_1, (f^\CC)_1)$ is a quasi-isomorphism.
Denote by $\mathsf{MHD}_\infty$ the category of $C_\infty$-mixed Hodge diagrams together with $\infty$-morphisms. Note that there is an inclusion $\iota : \mathsf{MHD} \hookrightarrow \mathsf{MHD}_\infty$.
\vspace{1ex}
\par Denote by $\mathrm{co}Lie\textrm{-}\mathsf{MHD}$ the category of coLie coalgebras in $\MHC$. As in the case of cdga's, there is a bar-cobar adjunction for mixed Hodge diagrams:
\[ \Omega : \mathsf{coLie}\textrm{-}\mathsf{MHD} \leftrightharpoons \mathsf{MHD} : \mathrm{B} \]
\begin{rema}
  This adjunction is not well-defined for mixed Hodge complexes as we defined them above, but only for a larger version of mixed Hodge complexes (denoted by $\MHC^{\NN\textrm{-fil}}$ in \cite{MHF}), where the weight and Hodge filtrations are not necessarily regular but satisfy a more relaxed regularity condition. This is because the bar and cobar functors naturally give non-regular filtrations. In the rest of this work we implicitly assume this larger category when we write $\MHC$, $\MHD$, $\mathsf{coLie}\textrm{-}\mathsf{MHD}$ or $\MHD_\infty$.
\end{rema}
Moreover, there is an extension of the bar functor to $\mathsf{MHD}_\infty$:
\[ \tilde{\mathrm{B}} : \mathsf{MHD}_\infty \to \mathsf{coLie}\textrm{-}\mathsf{MHD}\]
and an adjunction 
\[ \Omega \tilde{\mathrm{B}} : \mathsf{MHD}_\infty \leftrightharpoons \mathsf{MHD} : \iota. \]
The functors $\Omega \tilde{\mathrm{B}}$ and $i$ preserve quasi-isomorphisms and thus induce functors between the categories localized at quasi-isomorphisms, denoted $\Ho(\mathsf{MHD})$ and $\Ho(\mathsf{MHD}_\infty)$. Moreover, the unit
\[ A \to \Omega \tilde{\mathrm{B}}(A) \]
is an $\infty$-quasi-isomorphism for all $A \in C_\infty\textrm{-}\mathsf{MHD}$. As a consequence, we have

\begin{theo}[\cite{MHF} Theorem $3.31$]
  \label{equivCinfMHD-theo}
  The adjunction $\Omega \tilde{\mathrm{B}} : \mathsf{MHD}_\infty \leftrightharpoons \mathsf{MHD} : \iota$ induces an equivalence of categories 
  \[ \Ho(\mathsf{MHD}_\infty) \simeq \Ho(\mathsf{MHD}). \]
\end{theo}
\begin{rema}
  \label{mhdequivlength-rema}
  Note that this is an equivalence between $C_\infty$-mixed Hodge diagrams and mixed Hodge diagrams of a fixed length $s$. However, for the purposes of studying formality, the length is irrelevant. Indeed, as a consequence of Proposition \ref{Pmhalgmodels-prop} below, any mixed Hodge diagram of length $s$ is quasi-isomorphic to one of length $1$. Moreover, there is a bijection between isomorphism classes in $\Ho(\MHD_s)$ and $\Ho(\MHD_1)$ (see Proposition $3.44$ of \cite{MHF}).
\end{rema}
There is also a version of the homotopy transfer theorem for mixed Hodge diagrams:
\begin{theo}[\cite{MHF} Theorem $3.39$]
  \label{htt-theo}
  Given $A \in \mathsf{MHD}$, there is a $C_\infty$-mixed Hodge diagram isomorphic to $A$ in $\Ho(\mathsf{MHD}_\infty)$ whose underlying mixed Hodge complex is $H^*(A)$.
\end{theo}
As a consequence, we show that mixed Hodge diagrams admit models which are cdga's in mixed Hodge structures. This result was obtained in \cite{CG1}, where Cirici and Guillén adapted Sullivan's theory of minimal models of dga's to the setting of mixed Hodge diagrams. In previous work, Morgan \cite{morgan} also computed models of mixed Hodge diagrams but which are not themselves mixed Hodge diagrams.

\begin{prop}
  \label{Pmhalgmodels-prop}
  Any mixed Hodge diagram is quasi-isomorphic to a cdga in mixed Hodge structures.
\end{prop}

\begin{proof}
  By Theorem \ref{htt-theo}, there exists a $C_\infty$-mixed Hodge structure on $H^*(A)$ which is $\infty$-quasi-isomorphic to $A$. Let us denote this $C_\infty$-mixed Hodge diagram by $(H^*(A),m)$. Moreover, the unit 
  \[ (H^*(A),m) \to \Omega \tilde{\mathrm{B}}((H^*(A),m)) \]
  is an $\infty$-quasi-isomorphism. By Theorem \ref{equivCinfMHD-theo}, $A$ is thus quasi-isomorphic to 
  \[ \Omega \tilde{\mathrm{B}}((H^*(A),m)). \] 
  Denote by $\hat{\varphi}^u$ the comparison morphisms of $(H^*(A),m)$. Since $H^*(A)$ has trivial differential and $\hat{\varphi}^u$ are quasi-isomorphisms, it follows that they are $\infty$-isomorphisms. Therefore, the morphisms $\Omega \tilde{\mathrm{B}}(\hat{\varphi}^u)$ are isomorphisms of cdga's. Denote by 
  \[ \Omega \tilde{\mathrm{B}}(\hat{\varphi}) : \Omega \tilde{\mathrm{B}}(H^*(A_\Bbbk)) \otimes \CC \to  \Omega \tilde{\mathrm{B}}(H^*(A_\CC))  \]
  the isomorphism obtained by composing, and inverting when necessary, all the comparison morphisms. Denote also by $W$ and $F$ the filtrations on $\Omega \tilde{\mathrm{B}}(H^*(A))$ induced by the weight and Hodge filtrations of $H^*(A)$ by extending freely to the tensor algebra. Then, the cdga in mixed Hodge structures
  \[ (\Omega \tilde{\mathrm{B}}(H^*(A)),W, \Omega \tilde{\mathrm{B}}(\hat{\varphi})^{-1}(F)) \] 
  is quasi-isomorphic to $A$.
\end{proof}
\begin{rema}
  \label{Mhalgmodel-rema}
  Note that the model of $A \in \MHD$ in cdga in mixed Hodge structures constructed in this proof has underlying cdga $\Omega \tilde{\mathrm{B}}(H^*(A_\Bbbk))$. Its weight filtration is just the one of $H^*(A)$ freely extended. However, its Hodge filtration is the extension of the Hodge filtration of $H^*(A)$ but twisted by an isomorphism of cdga's, the comparison morphism $\Omega \tilde{\mathrm{B}}(\hat{\varphi})$.
\end{rema}

\section{Mixed Hodge formality}
In this section we define mixed Hodge formality and review a theory of obstructions to mixed Hodge formality constructed in \cite{MHF}. 
\subsection{Formal mixed Hodge diagrams}
Given any category $\Cc$ with a class of weak equivalences and a cohomology functor $H^* : \Cc \to \Cc$, it is natural to ask when an object $A \in \Cc$ is \textit{formal}. That is, when there exists a zig-zag of weak equivalences between $A$ and $H^*(A)$:
\[ A \xleftarrow{\sim} \cdots \xrightarrow{\sim} H^*(A). \]
A classical example is that of cdgas over a ring $\Bbbk$ with quasi-isomorphisms as weak equivalences. When $\Bbbk$ is a field, every chain complex is formal, whereas there are cdga's which are not formal. Formality in this setting is thus really a question about the multiplicative structure. The same question can be posed for the category of mixed Hodge diagrams with quasi-isomorphisms. 
\begin{defi}
  A mixed Hodge diagram $A$ is said to be \textit{formal} if there exists a zig-zag of quasi-isomorphisms of mixed Hodge diagrams:
  \[ A \xleftarrow{\sim} \cdots \xrightarrow{\sim} H^*(A). \]
\end{defi}
Forgetting the product structure yields an analogous notion of formality for mixed Hodge complexes. As shown in Theorem $4.8$ of \cite{CG2}, every mixed Hodge complex is formal. But not every mixed Hodge diagram is formal (see, for instance, Example \ref{nontrivDBprod-exam} below). Hence, just as in the cdga case, formality of mixed Hodge diagrams is intimately tied to the multiplicative structure. Furthermore, note that if a mixed Hodge diagram $A$ is formal, then $(A_\Bbbk,W)$ and $(A_\CC,W,F)$ are formal filtered and bifiltered cdgas, respectively. The converse, however, does not hold (see, for instance, Example $4.22$ of \cite{MHF}). Formality of mixed Hodge diagrams is thus not only a question about the multiplicative structure but how it interacts with both the $\Bbbk$-structure and the $\CC$-structure together with the Hodge filtration.

\subsection{Mixed Hodge Harrison cohomologies}
The obstructions to formality live in a mixed Hodge version of Harrison cohomology, which we now define. 

Let $H$ be a commutative algebra in $\mathrm{Mod}_\Bbbk$. The \textit{Harrison cohomology} of $H$ is the homology of the complex
\[ C^n(H) = \Hom_\Bbbk( H^{\otimes n}/\mathrm{Sh}, H), \]
where $\mathrm{Sh}$ denotes the sum of the $(p,q)$-shuffle products for all $p+q = n$ (see \cite{harrisoncoho}). Its differential is 
\begin{align*}
  \delta(f)( x_1 \otimes \cdots \otimes x_{n+1} ) = \,& x_1 \cdot f(x_2 \otimes \cdots \otimes x_n) + \sum_{i = 1}^n (-1)^i f(x_1 \otimes \cdots \otimes x_i \cdot x_{i+1} \otimes \cdots \otimes x_n) \\
  &+ (-1)^{n+1} (x_1 \otimes \cdots \otimes x_{n-1}).
\end{align*} 
Suppose now that $H$ is a commutative algebra in $\MHS$. The filtrations of $H$ induce mixed Hodge structures on $C^*(H)$ according to equation (\ref{homfil-eq}). Note that $\delta$ is a map of mixed Hodge structures so it induces a map
\begin{equation}
  \label{Harrextdif-eq}
  \mathrm{Ext}_{\mathsf{MHS}}^1(H^{\otimes n} / \mathrm{Sh},H) \xrightarrow{\delta} \mathrm{Ext}_{\mathsf{MHS}}^1(H^{\otimes n+1} / \mathrm{Sh},H). 
\end{equation} 
Denote by $C_{\mathrm{Ext}}^*(H)$ the resulting complex.
\begin{defi}
  \label{hodgeharr-coho-defi}
  Let $H$ be a commutative algebra in $\MHS$.
  \begin{enumerate}
    \item The \textit{$\mathrm{Ext}$-Harrison cohomology} of $H$, denoted by $\mathrm{Harr}_{\mathrm{Ext}}^*(H)$, is the homology of $C_{\mathrm{Ext}}^*(H)$ and
    \item The \textit{Deligne-Beilinson-Harrison cohomology} of $H$, denoted by $\mathrm{Harr}_{\mathrm{DB}}^*(H)$, is the homology of the cone
    \[ C_{\mathrm{DB}}^*(H) = \mathrm{Cone}\big(W_0 C^*(H) \oplus W_0 F^0 C^*(H) \otimes \CC \xrightarrow{\iota_\Bbbk - \iota_F} W_0 C^*(H) \otimes \CC \big), \]
    where $\iota_\Bbbk : W_0 C^*(H) \to W_0 C^*(H) \otimes \CC$ and $\iota_F : W_0 F^0 C^*(H) \otimes \CC \to W_0 C^*(H) \otimes \CC$ are the inclusions.
  \end{enumerate}
\end{defi}

\begin{rema}
  If $H$ is graded, then the grading induces an extra grading on $\mathrm{Harr}^{**}(H)$. And Likewise for $\mathrm{Harr}_{\mathrm{Ext}}^{**}(H)$ and $\mathrm{Harr}_{\mathrm{DB}}^{**}(H)$. Let us denote the grading induced by $H$ the homological grading and the Harrison cohomology grading by arity. So an element in $\mathrm{Harr}^{k,n}(H)$ is represented by a morphism $f : H^{\otimes k} \to H$ of homological degree $n$.
\end{rema}

\begin{rema}
  The notion of Deligne-Beilinson-Harrison cohomology is precisely the Deligne Beilinson cohomology (introduced by \cite{Beilinson} and defined in section \ref{DBcoho-sec}) of the complex of mixed Hodge structures $C^*(H)$.
\end{rema}

\subsection{Obstructions to mixed Hodge formality}
We now recall a theory of obstructions to mixed Hodge formality.
\begin{rema}
  In the theorem below, by a sequence of successively defined classes, we mean a sequence of classes such that each class is defined if all previous classes are defined and trivial.
\end{rema}
\begin{theo}[\cite{MHF} Theorem $4.17$]
  \label{obst-theo}
  Let $A$ be a mixed Hodge diagram. There is a sequence of successively defined obstruction classes 
  \[ \theta_k \in \mathrm{Harr}_{DB}^{k,2-k}(H^*(A)) \qquad \text{for } k \geq 3, \]
  such that if all obstructions are trivial then $A$ is formal.
\end{theo}
We give a description of representatives of the obstructions to formality. Suppose that $A$ is a mixed Hodge diagram of length $1$ (see Remark \ref{mhdequivlength-rema}). Let $(H^*(A),m)$ be a $C_\infty$-mixed Hodge diagram quasi-isomorphic to $A$ (such a $C_\infty$-diagram exists by Theorem \ref{htt-theo}). Recall that $(H^*(A),m)$ is composed of two $\infty$-algebras $m^\Bbbk$ and $m_\CC$:
\[ m^\Bbbk_n : H^*(A_\Bbbk)^{\otimes n} \to H^*(A_\Bbbk), \qquad m^\CC_n : H^*(A_\CC)^{\otimes n} \to H^*(A_\CC) \]
and an $\infty$-morphism $\hat{\varphi}$ (the comparison morphism):
\[ \hat{\varphi}_n : (H^*(A_\Bbbk) \otimes \CC)^{\otimes n} \to H^*(A_\CC). \]
We can further suppose that $H^*(A_\Bbbk) \otimes \CC = H^*(A_\CC)$ and $\hat{\varphi}_1 = id$.
\begin{prop}
  \label{obsrep-prop}
  Let $k \geq 3$ be the first value for which the triple
  \[ (m^\Bbbk_k,m^\CC_k,\varphi_{k-1}) \]
  is non-zero.  Then, the obstructions $\theta_n$ are trivial for $n \leq k-1$, so $\theta_k$ is defined and it is given by
\[ \theta_k = [(m^\Bbbk)_k,(m^\CC)_k, \hat{\varphi}_{k-1}]. \]
\end{prop}
Note that the non-vanishing of an obstruction does not \textit{a priori} imply non-formality of $A$. This is because the construction of an obstruction depends on choices of elements witnessing the vanishing of the previous obstructions. Thus, an obstruction might not vanish for a certain choice of such elements but vanish for others. However, the first obstruction - and in certain settings, the second one - do not depend on these choices.

\begin{defi}
  We say that a graded algebra $H$ in $\MHS$ is \textit{pure} if $H^n$ is a pure Hodge structure of weight $n$ for each $n \in \ZZ$.
\end{defi}

\begin{prop}[\cite{MHF} Proposition $4.18$]
    \label{firstobswelldef-prop}
    Let $A \in \MHD$. The first obstruction $\theta_3$ to formality of $A$ does not depend on choices. Furthermore, suppose that one of the following conditions holds:
    \begin{enumerate}
        \item $H^*(A)$ is concentrated in even degrees or
        \item $A$ is a real mixed Hodge diagram and $H^*(A)$ is pure.
    \end{enumerate}
    Then, $\theta_3$ is trivial so the second obstruction $\theta_4$ is defined and this obstruction does not depend on choices either.      
\end{prop}
As a consequence, we have:
\begin{coro}
  \label{firstobsnonvanish-coro}
    If the first obstruction $\theta_3$ to formality of $A \in \Pp\textrm{-}\MHD$ is non-trivial, then $A$ is not formal. Similarly, if $A$ satisfies one of the conditions of the previous proposition and $\theta_4$ is non-trivial, then it follows that $A$ is not formal.
\end{coro}

When the cohomology of a mixed Hodge diagram is pure, the obstructions of Theorem \ref{obst-theo} get simplified.
\begin{theo}
  \label{obst-coro}
  Let $A \in \mathsf{MHD}$ be such that $H^*(A)$ is pure. Then, there exist successively defined classes 
  \[\phi_k \in \mathrm{Harr}_{\mathrm{Ext}}^{k,1-k}(H^*(A)) \quad \text{for } k \geq 2, \]
  such that if all classes are trivial then $A$ is formal.
\end{theo}
\begin{proof}
  Let $A \in \mathsf{MHD}$ be such that $H = H^*(A)$ is pure. Then, the natural map
\begin{align}
    \label{harrextmap-eq}
    \mathrm{Harr}_{\mathrm{DB}}^{k,m}(H) &\xrightarrow{\pi} \mathrm{Harr}_{\mathrm{Ext}}^{k-1,m}(H) \\ \nonumber
    [f_\Bbbk,f_\CC,h] &\mapsto [h]
  \end{align}
is injective for $m \neq 0$. The result is then a direct consequence of Theorem \ref{obst-theo}.
\end{proof}

\section{Deligne-Beilinson cohomology}
\label{DBcoho-sec}
In this section, we review the notion of Deligne-Beilinson cohomology of a mixed Hodge complex $A$, as defined by Beilinson in \cite{Beilinson}. We review two complexes $C_\mathrm{DB}(A)$ and $P_\mathrm{DB}(A)$ that compute this cohomology, the latter one being a cdga when $A$ is a mixed Hodge diagram. As an extra result, which we do not use in the rest of this work, we compute the homotopy transfer of the product in $P_\mathrm{DB}(A)$ to $C_\mathrm{DB}(A)$.

Given $A \in \MHC$ with comparison morphisms
\[
  \begin{tikzcd}[column sep = tiny]
    & (A_1,W) & & \cdots & \\
    (A_0,W) = (A_\Bbbk,W) \otimes \CC \arrow[ru,"\varphi_{01}"] & & (A_2,W) \arrow[lu,"\varphi_{21}"'] \arrow[ru,"\varphi_{23}"] & & \arrow[lu,"\varphi_{s s-1}"'] (A_s,W) = (A_\CC,W)  
  \end{tikzcd}   \]
we use the following notation for the sum of the complexes on the top and on the bottom of this diagram:
\begin{align*}
  &\Sigma_{\mathrm{bot}} := W_0 A_\Bbbk \oplus F^0 W_0 A_\CC \bigoplus_{1 \leq i \leq s/2-1} W_0 A_{2i}, \\
  &\Sigma_{\mathrm{top}} := \bigoplus_{1 \leq j \leq s/2} W_0 A_{2j-1}.
\end{align*} 
Here, we see $W_0 A_i$, $\Sigma_{\mathrm{bot}}$ and $\Sigma_{\mathrm{top}}$ as $\Bbbk$-vector spaces.
\begin{defi}
  \label{Delcmpx-defi}
  Given $A \in \MHC$, the \textit{$0$-th Deligne-Beilinson complex} of $A$ is the cochain complex
  \[ C^*(A,0)_{\mathrm{DB}} := \mathrm{Cone}( \Sigma_{\mathrm{bot}} \xrightarrow{\bigoplus \varphi_{u} - \varphi_{u'}} \Sigma_{\mathrm{top}}), \]
  where the sum ranges over successive comparison morphisms $u$ and $u'$ with the same target. 
\end{defi}
\begin{exam}
  For instance, when $s=2$, a mixed Hodge complex $A$ has shape
\[ \begin{tikzcd}[column sep = small]
    & (A_1,W) &  \\
    (A_\Bbbk,W) \otimes \CC \arrow[ru,"\varphi_{01}"] & & \arrow[lu,"\varphi_{21}"'] (A_\CC,W),
  \end{tikzcd} \]
and the complex $C^*(A,0)_{\mathrm{DB}}$ is given by
\[ C^*(A,0)_{\mathrm{DB}} = \mathrm{Cone}( W_0 A_\Bbbk \oplus F^0 W_0 A_\CC \xrightarrow{\varphi_{01} - \varphi_{21}} W_0 A_1). \]
\end{exam}

\begin{defi}
  The \textit{Tate twist} $\Bbbk(p)$ is the pure Hodge structure on $(2 \pi i)^p \Bbbk$ concentrated in degree $0$, weight $-2p$ and with \[ F^p = \CC \supset F^{p+1} = 0. \]
\end{defi}
For $A \in \MHC$, denote $A(p) := A \otimes \Bbbk(p)$.
\begin{defi}
  For $p \in \ZZ$, the \textit{$p$-th Deligne-Beilinson complex} is given by
  \[ C^*(A,p)_{\mathrm{DB}} := C^*(A(p)[2p],0)_{\mathrm{DB}}.\] 
  The \textit{$p$-th Deligne-Beilinson cohomology} of $A$, denoted by $H_{\mathrm{DB}}^*(A,p)$, is the cohomology of $C(A;p)_{\mathrm{DB}}$.
\end{defi}

Summing the Deligne-Beilinson complexes over all $p \in \ZZ$ yields a complex
\[
  C^*(A)_{\mathrm{DB}} := \bigoplus_{p \in \ZZ} C^*(A,p)_{\mathrm{DB}},
\]
whose cohomology, the Deligne-Beilinson cohomology of $A$, is denoted by $H^*_{\mathrm{DB}}(A)$. This gives a functor 
\begin{align}
  \label{absHodgecpx-eq}
  C^*(-)_{\mathrm{DB}}: \MHC &\to \ch{\Bbbk} \\
  A &\mapsto C^*(A)_{\mathrm{DB}} \nonumber
\end{align}
\begin{rema}
  Technically, we have defined the Deligne-Beilinson complex for mixed Hodge complexes of even length $s$. For odd $s$, we define the Deligne-Beilinson complex of $A \in \MHC_s$, as the Deligne-Beilinson complex of its inclusion into $\MHC_{s+1}$ by inserting an identity as the last comparison morphism.
\end{rema}  
If $A \in \mathsf{MHD}$, then for each $\alpha \in [0,1]$, the complex $C^*(A)_{\mathrm{DB}}$ admits a product 
\[ \cdot_\alpha : C^*(A,p)_{\mathrm{DB}} \otimes C^*(A,q)_{\mathrm{DB}} \to C^*(A,p+q)_{\mathrm{DB}}, \]
(see section $1$ of \cite{Beilinson}). It is given by
\begin{align}
  \label{Del-prod}
  ((a_i)_{\mathrm{bot}}, (a_j)_{\mathrm{top}}) \otimes \big((b_i)_{\mathrm{bot}}, (b_j)_{\mathrm{top}}) \mapsto  ((a_i \cdot b_i)_{\mathrm{bot}}, (\theta_j)_{\mathrm{top}}\big),
\end{align} 
where 
\begin{align*}
  \theta_j = &\alpha (a_j \cdot \varphi_{u}(b_{j-1}) + (-1)^{|a|} \varphi_{u'}(a_{j+1}) \cdot b_j) \\ &+ (1- \alpha)( a_j \cdot \varphi_{u'}(b_{j+1}) + (-1)^{|a|} \varphi_{u}(a_{j-1}) \cdot b_j).
\end{align*} 
Here, for each $j$ indexing an element in $\Sigma_{\mathrm{top}}$, $\varphi_u = \varphi_{j-1,j}$ and $\varphi_{u'} = \varphi_{j,j+1}$ are the comparison morphisms with target $A_j$. 
\begin{exam}
  For instance, when $s=2$, the formulas become
\begin{align*}
  (a_\Bbbk,a_\CC,a_h) &\otimes \big(b_\Bbbk,b_\CC,b_h) \mapsto  (a_\Bbbk \cdot b_\Bbbk, a_\CC \cdot b_\CC, \, \\ &\alpha (a_h \cdot \varphi_{01}(b_\Bbbk) + (-1)^{|a|} \varphi_{21}(a_\CC) \cdot b_h)  \\
  &+ (1- \alpha)( a_h \cdot \varphi_{21}(b_{\CC}) + (-1)^{|a|} \varphi_{01}(a_\Bbbk) \cdot b_h) \big).
\end{align*}
\end{exam}
For $\alpha = 0$ and $\alpha = 1$ the product is associative and for $\alpha = 1/2$ it is graded-commutative. For all values of $\alpha$, the products are homotopy equivalent, so the induced products on cohomology are equal (see section $1$ of \cite{Beilinson}). Thus, $H^*_{\mathrm{DB}}(A)$ inherits an associative and graded-commutative product.
\vspace{1ex}
\par Although $H^*_{\mathrm{DB}}(A)$ is a cdga, there is no natural associative and commutative product on $C^*(A)_\mathrm{DB}$. There is, however, a larger complex quasi-isomorphic to it which does admit the structure of a cdga. Given $A \in \MHC$, let
\[ \tilde{\Sigma}_{\mathrm{top}} = \bigoplus_{1 \leq j \leq s/2} \Lambda(t,dt) \otimes W_0 A_{2i-1}. \]
\begin{defi}
  Given $A \in \mathsf{MHC}$, define the complex
  \begin{align*}
      P_{\mathrm{DB}}(A,0) = \{ ((a_i)_{\mathrm{bot}}, (\omega_j)_{\mathrm{top}}) \in \Sigma_{\mathrm{bot}} \times \tilde{\Sigma}_{\mathrm{top}} \, | \, &(\omega_j)_0 = \varphi_u(a_{j-1}), \\ &(\omega_{j})_1 = \varphi_{u'}(a_{j+1}) \},
  \end{align*}
  where $u : j-1 \to j$, $u' : j+1 \to j$ and $(\omega_j)_0$ and $(\omega_{j})_1$ are the evaluations at $t=0$ and $t=1$, respectively. Given $p \in \ZZ$, define
  \[ P_{\mathrm{DB}}(A,p) := P_{\mathrm{DB}}(A(p)[2p],0). \]
\end{defi}
If $A$ is a mixed Hodge diagram, then there is a natural product
\[ P_{\mathrm{DB}}(A,p) \otimes P_\mathrm{DB}(A,q) \to P_\mathrm{DB}(A,p+q) \]
induced by the products of $A$ and of $\Lambda(t,dt)$. Thus, for $A \in \mathsf{MHD}$, the direct sum
\begin{equation}
  \label{PDBdefi-eq}
  P_\mathrm{DB}(A) = \bigoplus_{p \in \ZZ} P_\mathrm{DB}(A,p)
\end{equation} 
is a cdga. This defines a functor
\begin{align*}
  P_\mathrm{DB}(-) : \mathsf{MHD} &\to \mathsf{cdga}_\Bbbk \\
  A &\mapsto P_\mathrm{DB}(A).
\end{align*} 
\begin{rema}
  This complex was considered before, for instance, in \cite{highregDB-thesis} and \cite{BeilRegSp}, for mixed Hodge complexes of length $2$ and identity comparison morphisms.
\end{rema}
The complexes $P_\mathrm{DB}(A)$ and $C_{\mathrm{DB}}^*(A)$ are homotopy equivalent. The following is proven for instance in section $3.3$ of \cite{highregDB-thesis}.
\begin{lemm}
  \label{Delcomp-equiv-lemm}
  For each $p \in \ZZ$, there is a contraction
  \[
    \begin{tikzcd}
      P_{\mathrm{DB}}(A,p) \arrow[loop left, "\eta(p)"] \arrow[r,shift left,"ev(p)"] \arrow[r,shift right, leftarrow, "l(p)"'] & C_{\mathrm{DB}}^*(A,p), 
    \end{tikzcd} 
  \]
  with morphisms defined by 
  \begin{align*}
    &ev(p)((a_i)_{\mathrm{bot}},(\omega_j)_\mathrm{top}) =  \bigg((a_i)_\mathrm{bot}, \bigg(\int_0^1 \omega_j \bigg)_\mathrm{top} \, \bigg), \\
    &s(p)((a_i)_{\mathrm{bot}},(a_j)_\mathrm{top}) = ((a_i)_{\mathrm{bot}}, (t \varphi_{u'}(a_{j+1}) + (1-t) \varphi_u(a_{j-1}) + dt a_j)_\mathrm{top}) \\
    &\eta(p)((a_i)_{\mathrm{bot}},(\omega_j)_\mathrm{top}) = \bigg( (0)_{\mathrm{bot}}, \bigg(t \int_0^1 \omega_j - \int_0^t \omega_j \bigg)_\mathrm{top} \bigg).
  \end{align*}
  More precisely, these maps satisfy the equation 
  \[ d \eta + \eta d = l(p) ev(p) - id. \]
  Moreover, these also satisfy 
  \[ ev(p)  l(p) = id, \quad \eta(p)^2 = ev(p) \eta(p) = \eta(p)l(p) = 0. \] 
\end{lemm}
The morphisms $ev = \bigoplus_{p \geq 0} ev(p)$ and $s = \bigoplus_{p \geq 0} s(p)$ are natural homotopy equivalences $P_\mathrm{DB}(-) \Rightarrow C^*_\mathrm{DB}(-)$ and $C^*_\mathrm{DB}(-) \Rightarrow P_\mathrm{DB}(-)$. As an independent result, we compute the homotopy transfer of the commutative product of $P_\mathrm{DB}(A)$ to $C_{\mathrm{DB}}^*(A)$ through the contraction just defined. This will not be used in the rest of the work.
\begin{prop}
  \label{PDB-transf-prop}
  Given $A \in \MHD$, the $C_\infty$-structure on $C_{\mathrm{DB}}^*(A)$ given by homotopy transfer of the product of $P_\mathrm{DB}(A)$ has the following structure morphisms. The product $m_2$ is the $\alpha$-product $\cdot_{1/2}$ for $\alpha = 1/2$ (see (\ref{Del-prod})). Given $k\geq 3$ and 
  \[ x^m = ((x^m_i)_\mathrm{bot}, (x^m_j)_\mathrm{top}) \in C_{\mathrm{DB}}^*(A), \] 
  we have
  \[
    m_k(x^1,\dots,x^k) = \bigg( (0)_\mathrm{bot}, \bigg( \frac{(-1)^{k+1}}{k!} B_{k+1}^-\cdot e_{k,j}(x^1, \dots x^k) \bigg)_\mathrm{top} \, \bigg).
  \]
  Here, $B_k^-$ denotes the $k$-th Bernoulli number and the element
  \[e_{k,j}(x^1,...,x^k) \in A_{2j-1} \] 
  is given by
  \begin{align*}
      e_{k,j}(x^1, \dots, x^k) = \sum_{m = 1}^{k} &(-1)^{\epsilon_{k,m}(x^1,\dots,x^k)}  {{k-1}\choose{m-1}} \cdot \\ &x^1_j \cdots x^{m-1}_j  (\varphi_u(x^m_{j-1}) - \varphi_{u'}(x^m_{j+1})) x^{m+1}_j \cdots x^k_j
  \end{align*} 
  where $\epsilon_{k,m}(x^1,\dots,x^k)$ is defined recursively:
  \begin{align*}
    &\epsilon_{k,k}(x^1,\dots,x^k) = \sum_{s = 2}^{k-1} (s+1)|x^{k-s}|, \quad \epsilon_{k,m}(x^1,\dots,x^k) = 0 \quad \text{for } k < m\\
    &\epsilon_{k,m}(x^1,\dots,x^k) = \epsilon_{k-1,m}(x^1,\dots,x^{k-1}) + \sum_{s=1}^{k-1} |x^s| \quad \text{for } k > m.
  \end{align*}
\end{prop}
\begin{proof}
  We use the formulas of the homotopy transfer theorem for associative algebras of \cite{Markl-trans-ass}. Since $P_\mathrm{DB}(A)$ is a commutative product, the resulting transferred structure is a $C_\infty$-structure (see \cite{transCinf}). Denote the product of $P_\mathrm{DB}(A)$ by $\mu$. The structure maps are thus given by
  \[ m_k = ev \circ \mathfrak{p}_k \circ l^{\otimes k}. \]
  The maps $\mathfrak{p}_k$ are the $\mathfrak{p}$-kernels and are defined recursively by setting $\mathfrak{p}_2 = \mu$ and
  \[ \mathfrak{p}_k = \sum_{i = 1}^k (-1)^{i+1} \mu(\eta \circ \mathfrak{p}_i,\eta \circ \mathfrak{p}_{k-i}), \]
  where, by convention, $\eta \mathfrak{p}_1 = id$. It is an easy check that
  \[ m_2 = ev \circ \mu \circ l^{\otimes 2} = \cdot_{1/2}.\]
  For the higher structure maps, we first prove, by induction, that for all $k \geq 2$,
  \begin{equation}
    \label{etapkernel-eq}
    \eta \circ \mathfrak{p}_k \circ l^{\otimes k}(x^1,\dots,x^k) = ((0)_{\mathrm{bot}}, (b_{k}(t) \cdot e_{k,j}(x^1,\dots,x^k))_\mathrm{top}),
  \end{equation} 
  where $b_k(t)$ is the polynomial
  \begin{equation}
    \label{Bernpoly-eq}
    b_{k}(t) = \frac{(-1)^k}{k!}\sum_{s=0}^{k-1} {{k}\choose{s}} B^-_{s} t^{k-s}
  \end{equation}
  and $e_{k,j}(x^1,\dots,x^k) \in A_{2j-1}$ is the element defined in the statement. Note that $\eta$ has trivial bottom components. We focus on the top component that lives in $A_{2j-1}$ for a fixed value of $j$. The projection of $\mathfrak{p}_2 l^{\otimes 2}(x^1,x^2)$ to this component yields
  \[ \big(t \varphi_{u'}(x^1_{j+1}) + (1-t) \varphi_u(x^1_{j-1}) + dt x^1_j \big)\big(t \varphi_{u'}(x^2_{j+1}) + (1-t) \varphi_u(x^2_{j-1}) + dt x^2_j \big). \]
  The terms with $dt$ are
  \begin{align*}
      dt \Big( &x^1_j \big(t \varphi_{u'}(x^2_{j+1}) + (1-t) \varphi_u(x^2_{j-1}) \big) + \\ &(-1)^{|x_1|}\big(t \varphi_{u'}(x^1_{j+1}) + (1-t) \varphi_u(x^1_{j-1})\big)x^2_j \Big).
  \end{align*} 
  Hence, applying $\eta$ and rearranging the terms, we arrive at
  \[ \bigg( \frac{t^2}{2} - \frac{t}{2} \bigg) \Big( x^1_j \big(\varphi_u(x^2_{j-1}) - \varphi_{u'}(x^2_{j+1}) \big) + (-1)^{|x^1|} \big(\varphi_u(x^1_{j-1}) - \varphi_{u'}(x^1_{j+1}) \big)x^2_j \Big). \]
  This coincides with the given formula for $k = 2$. Assume the induction hypothesis for all $j \leq k-1$. Then, the only terms in the formula
  \[ \mathfrak{p}_k l^{\otimes k}(x^1,\dots,x^k) = \sum_{i = 1}^k (-1)^{i+1} \mu(\eta \circ \mathfrak{p}_i \circ l^{\otimes i},\eta \circ \mathfrak{p}_{k-i} \circ l^{\otimes k-i})(x^1,\dots,x^k) \]
  with a $dt$ have $j$ component given by
  \begin{align*}
     dt \big((-1)^{k|x^1|}x^1_j & \cdot \eta \circ \mathfrak{p}_{k-1} \circ l^{\otimes k-1}(x^2,\dots,x^k) \\ &+ (-1)^{\sum_{s=1}^{k-1} |x^s|} \eta \circ \mathfrak{p}_{k-1} \circ l^{\otimes k-1}(x^1,\dots,x^{k-1}) \cdot x^k_j \big).
  \end{align*}
  This is equal to
  \begin{align*}
    dt \big( (-1)^{k|x^1|} b_{k-1}(t) &\cdot x^1_j \cdot e_{k-1,j}(x^2,\dots,x^k)   \\ &+ (-1)^{\sum_{s=1}^{k-1} |x^s|} b_{k-1}(t) \cdot e_{k-1,j}(x^1,\dots,x^{k-1}) \cdot x^k_j \big).
  \end{align*}
  Note that 
  \begin{align*}
     e_{k,j}(x^1,\dots,x^k) = (-1)^{k|x^1|} x^1_j &\cdot e_{k-1,j}(x^2,\dots,x^k) \\ &+ (-1)^{\sum_{s=1}^{k-1} |x^s|} e_{k-1,j}(x^1,\dots,x^{k-1}) \cdot x^k_j.
  \end{align*}
  This is because, for $1 < m < k$, the element
  \[ x^1_j \cdots x^{m-1}_j  (\varphi_u(x^m_{j-1}) - \varphi_{u'}(x^m_{j+1})) x^{m+1}_j \cdots x^k_j \]
  appears twice in the right side of the previous equality. Once with coefficient
  \[ (-1)^{\epsilon_{k-1,m}(x^1,\dots,x^{k-1}) + \sum_{s=1}^{k-1} |x^s|} {{k-2}\choose{m-1}} \]
  and once with coefficient
  \[ (-1)^{\epsilon_{k-1,m}(x^2,\dots,x^{k}) +k |x^1| } {{k-2}\choose{m-2}}. \]
  An induction argument shows that 
  \[ \epsilon_{k-1,m}(x^2,\dots,x^{k}) +k |x^1| = \epsilon_{k,m}(x^1,\dots,x^k). \]
  Hence the two terms have the same sign. The claim then follows from the identity
  \[ {{k-2}\choose{m-1}} + {{k-2}\choose{m-2}} = {{k-1}\choose{m-1}}. \]
  The cases $m=1$ and $m=k$ are a straightforward check. It then follows that the component of $\eta \circ \mathfrak{p}_k \circ l^{\otimes k}(x^1,\dots,x^k)$ living in $A_{2j-1}$ is given by
  \[ \bigg( t\int_0^1 b_{k-1}(t) - \int_{0}^{t} b_{k-1}(t) \bigg) e_{k,j}(x^1,\dots,x^k). \]
  The induction argument is proved after checking the equality
  \[ b_k(t) = t \int_{0}^1 b_{k-1}(t) - \int_0^t b_{k-1}(t), \]
  which follows from direct computation and the following identity of Bernoulli numbers:
  \begin{equation}
    \label{Bernident-eq}
     B^-_k = - \sum_{s = 0}^k {{k}\choose{s}} \frac{B^-_s}{k-s+1}.
  \end{equation}
  Using (\ref{etapkernel-eq}) and applying a similar reason as in the induction step, we see that the top component of 
  \[ ev \circ \mathfrak{p}_k \circ l^{\otimes k}(x^1,\dots,x^k) \]
  living in $A_{2j-1}$ is 
  \[ \int_0^1 b_{k}(t) \cdot e_{k,j}(x^1,\dots,x^k). \]
  Thus, the claim of the proposition follows from the fact that
  \[ \int_0^1 b_{k}(t) = \frac{(-1)^{k+1}}{k!} B_{k+1}^-. \]
  This is shown using again the identity (\ref{Bernident-eq}).
\end{proof}
As a simple consequence of the previous proposition, the products of Deligne-Beilinson cohomology induced by (\ref{Del-prod}) coincide with the product induced by $P_{\mathrm{DB}}(A)$.
\begin{coro}
  \label{ev-s-prod-lemm}
  The maps \[ev : P_\mathrm{DB}(A) \to C^*_\mathrm{DB}(A) \quad \text{and} \quad s  : C^*_\mathrm{DB}(A) \to P_\mathrm{DB}(A) \] induce isomorphisms of cohomology algebras.
\end{coro}
\begin{proof}
  This follows from Proposition \ref{PDB-transf-prop}, as $ev$ and $l$ are the first components of $\infty$-quasi-isomorphisms between $P_{\mathrm{DB}}(A)$ and a $C_\infty$-structure on $C_{\mathrm{DB}}(A)$ whose product $m_2$ is one of the products in (\ref{Del-prod}). Moreover, these products are all homotopy equivalent to each other.
\end{proof}

\begin{prop}
  \label{delcompqiso-prop}
  The functors \[ C^*_\mathrm{DB}(-) : \MHC \to \ch{\Bbbk} \quad \text{and} \quad P_\mathrm{DB}(-) : \mathsf{MHD} \to \mathsf{cdga}_\Bbbk \] preserve quasi-isomorphisms.
\end{prop}
\begin{proof}
  Let $g : A \to B$ be a quasi-isomorphism of mixed Hodge diagrams. It is sufficient to check that $C^*_\mathrm{DB}(g)$ is a quasi-isomorphism of complexes. By the definition of $C^*_\mathrm{DB}(-)$ as a cone, there is a commuting diagram whose horizontal arrows are short exact sequences
  \[ \begin{tikzcd}[column sep = small]
    0 \arrow[r] & \Sigma_{\mathrm{top}}(A(p)[2p])[1] \arrow[d] \arrow[r] & C^*_\mathrm{DB}(A,p) \arrow[r] \arrow[d,"C^*_\mathrm{DB}(g)"] & \Sigma_{\mathrm{bot}}(A(p)[2p]) \arrow[r] \arrow[d] & 0 \\
    0 \arrow[r] & \Sigma_{\mathrm{top}}(B(p)[2p])[1] \arrow[r] & C^*_\mathrm{DB}(B,p)\arrow[r] & \Sigma_{\mathrm{bot}}(B(p)[2p])  \arrow[r] & 0 
  \end{tikzcd} \]
  and the fact that $g$ is a quasi-isomorphism implies that the vertical arrows on the edges are quasi-isomorphisms. This implies that the middle arrow is also a quasi-isomorphism.
\end{proof}

The functors $C_\mathrm{DB}(-)$ and $P_\mathrm{DB}(-)$ are defined for mixed Hodge diagrams of a fixed length. As observed in Remark \ref{mhdequivlength-rema}, the inclusion 
\[ \MHD_s \hookrightarrow \MHD_{s+1} \]
induces an essentially surjective functor on homotopy categories. Denote by $C_{\mathrm{DB},s}(-)$ and $C_{\mathrm{DB},s+1}(-)$ the functors $C_\mathrm{DB}(-)$ for length $s$ and $s+1$, respectively. There is a natural quasi-isomorphism $C_{\mathrm{DB},s}(-) \Rightarrow C_{\mathrm{DB},s+1}(-)$ given by
\begin{align*}
    ((a_0,\dots,a_s)_\mathrm{bot}, (a_1,\dots,a_{s-1})_\mathrm{top}) \mapsto ((a_0,\dots,&a_s,a_s)_\mathrm{bot}, \\&(a_1,\dots,a_{s-1},0)_\mathrm{top}). 
\end{align*} 
For the functor $P_\mathrm{DB}(-)$, there is an analogous natural quasi-isomorphism from $P_{\mathrm{DB},s}(-)$ to  $P_{\mathrm{DB},s+1}(-)$:
\begin{align*}
    ((a_0,\dots,a_s)_\mathrm{bot}, (\omega_1,\dots,\omega_{s-1})_\mathrm{top}) \mapsto ((a_0,\dots,&a_s,a_s)_\mathrm{bot}, \\ &(\omega_1,\dots,\omega_{s-1},a_s)_\mathrm{top}).
\end{align*} 
Thus, given $A \in \MHD_s$, to compute $P_\mathrm{DB}(A) \in \mathsf{cdga}_\Bbbk$ up to quasi-isomorphism, one can first compute a model of $A$ of lower length (by Proposition \ref{Pmhalgmodels-prop}, there exists such a model) and then apply $P_\mathrm{DB}(-)$ to it.
\begin{rema}
\label{sesHDB-rema}
The short exact sequence of complexes
\[ 0 \to \Sigma_{\mathrm{top}}(A(p)[2p])[1] \to C^*_{\mathrm{DB}}(A,p) \to \Sigma_{\mathrm{bot}}(A(p)[2p]) \to 0\]
induces a long exact sequence in cohomology
\[ \cdots \to H^{*-1}(\Sigma_{\mathrm{top}}(A(p)[2p])) \to H^*_\mathrm{DB}(A,p) \to H^*(\Sigma_{\mathrm{bot}}(A(p)[2p])) \to \cdots \]
This breaks down into short exact sequences for all $n \in \ZZ$:
\[ 0 \to J^p H^{2p + n -1}(A_\Bbbk) \to H^n_\mathrm{DB}(A,p) \to \mathrm{Hdg}^{2p+n}(A,p) \to 0, \]
where 
\[ \mathrm{Hdg}^{2p+n}(A,p) := (2 \pi i)^p W_{2p} H^{2p+n}(A_\Bbbk) \cap F^p W_{2p} H^{2p + n}(A_\CC), \]
the vector space $J^p H^{2p + n-1}(A)$ is the \textit{$p$-th intermediate Jacobian} of $H^{2p + n-1}(A_\Bbbk)$:
\[ J^p H^{2p + n-1}(A_\Bbbk) := \frac{W_{2p} H^{2p + n-1}(A_\CC)}{(2 \pi i)^p W_{2p} H^{2p + n-1}(A_\Bbbk) + F^p W_{2p} H^{2p + n-1}(A_\CC)} \]
and we make the identification $H^*(A_\CC) \cong H^*(A_\Bbbk) \otimes \CC$ through the comparison morphisms (see, for instance, Corollary $2.10$ of \cite{DBEsnVie}).
\end{rema}

\begin{defi}
  \label{sqzeroext-defi}
  Define the algebra $\mathrm{Hdg}(A)$ by
  \[ \mathrm{Hdg}^*(A) :=\bigoplus_{p} \mathrm{Hdg}^{2p+*}(A,p), \]
  where the product is given by the product of $H^*(A_\Bbbk)$ and preserves both degrees $p$ and $n$. Define also the $\mathrm{Hdg}(A)$-module $JH^*(A)$ by 
  \[ JH^*(A) = \bigoplus_{p} J^p H^{2p + * - 1}(A_\Bbbk). \] 
\end{defi}
The following is the main property of $H_\mathrm{DB}^*(A)$ of interest in this paper.
\begin{lemm}
    \label{sqzeroext-lemm}
    The algebra $H_\mathrm{DB}^*(A)$ is a square-zero extension of $\mathrm{Hdg}(A)$:
    \[ 0 \to JH^*(A) \to H_\mathrm{DB}^*(A) \to \mathrm{Hdg}(A) \to 0.\]
\end{lemm}
\begin{proof}
    As observed in Remark \ref{sesHDB-rema}, the kernel of the projection
    \[  H^*_\mathrm{DB}(A)  \to \mathrm{Hdg}(A)\]
    is given by $JH^*(A)$. Thus, it is sufficient to prove that the product of two elements in the image of 
    \[ JH^*(A) \to H^*_\mathrm{DB}(A) \]
    is trivial. This is because this map is induced by the inclusion
\begin{align*}
  \Sigma_{\mathrm{top}}(A(p)[2p])[1] &\to C^*_{\mathrm{DB}}(A,p) \\
  (a_j)_{\mathrm{top}} &\mapsto ((0)_{\mathrm{bot}},(a_j)_{\mathrm{top}})
\end{align*}
and the product of elements of this form is trivial (see \ref{Del-prod}). 
\end{proof}
\section{Splitting of Deligne-Beilinson Cohomology}
\label{SplitDBcoho-sec}
In this section, we prove that the first obstruction to formality of $A \in \MHD$ is the obstruction to formality of $A$ seen as a commutative algebra in the $1$-category $D(\MHS) \simeq \Ho(\MHC)$ with the derived tensor product. As a corollary, we show that the vanishing of the first obstruction implies the splitting of $H^*_{\mathrm{DB}}(A)$ as a square-zero extension.
\subsection{Derived tensor product of mixed Hodge structures}
We give a description of the derived tensor product of the derived category of mixed Hodge structures $D(\MHS)$. The category $\ch{\MHS}$ admits a dg-enrichment whose homotopy category is equivalent to $D(\MHS)$. Given $A \in \ch{\MHS}$, the $0$-th Deligne-Beilinson complex of $A$ is given by
\[ C^*_{\mathrm{DB}}(A,0) = \mathrm{Cone}( W_0 A \oplus F^0 W_0 A \otimes \CC \xrightarrow{\iota_\Bbbk - \iota_\CC} W_0 A \otimes \CC), \]
where $\iota_\Bbbk$ and $\iota_\CC$ denote the respective inclusions. Recall that the category $\ch{\MHS}$ is closed symmetric monoidal, see (\ref{homfil-eq}). Using this, we define a dg-enrichment of $\ch{\MHS}$. Given, $A,B \in \ch{\MHS}$, let
\[ \Gamma(A,B) := C^*_\mathrm{DB}(\underline{\Hom}(A,B),0). \]
By the definition of the $\mathrm{Cone}$ and Proposition \ref{mh-ext}, we have that the cohomology in degree $0$ of $\Gamma(A,B)$ fits into a short exact sequence
\[ 0 \to \mathrm{Ext}^1_\MHS(H^*(A),H^*(B))^{-1} \to H^0(\Gamma(A,B)) \to \Hom_\MHS^0(H^*(A), H^*(B)). \]
The composition
\[ \Gamma(B,C) \otimes \Gamma(A,B) \to \Gamma(A,C) \]
is defined on $g = (g_\Bbbk,g_\CC, g_h)$ and $f = (f_\Bbbk,f_\CC, f_h)$ by
\[
  g \circ f := (g_\Bbbk \circ f_\Bbbk, g_\CC \circ f_\CC, g_h \circ f_\Bbbk + (-1)^{|g|} g_\CC \circ f_h). 
\]
Note that this corresponds to the product (\ref{Del-prod}) for $\alpha = 0$. The complexes $\Gamma(-,-)$ together with these composition maps make $\ch{\MHS}$ a dg-category, which we denote by $\ch{\MHS}^\Gamma$. Given a dg-category $\Dd$, there is an associated category $\Ho(\Dd)$ with the same objects as $\Dd$ and morphisms given by
\[ \Hom_{\Ho(\Dd)}(A,B) = H^0(\Hom_\Dd(A,B)), \qquad \text{for } A,B \in \Dd. \]
It is called the homotopy category of $\Dd$. To avoid confusion with the localization at quasi-isomorphisms, let us denote the homotopy category associated to $\ch{\MHS}^\Gamma$ by
\[ \Ho_{dg}\big(\ch{\MHS}^{\Gamma}\big) \] 
The arguments in section $3$ of \cite{Beilinson} (see also section $4$ of \cite{CG2}) prove that the identity of $\ch{\MHS}$ induces an equivalence
\[ D(\MHS) \simeq \Ho_{dg}\big(\ch{\MHS}^\Gamma \big). \]
\begin{rema}
  The notions of pre-morphism and ho-morphism of mixed Hodge complexes in section $3$ of \cite{CG2} coincide with the dg-enrichment $\Gamma(-,-)$ and its kernel in degree $0$, respectively.
\end{rema}
The symmetric monoidal tensor of $\MHS$ makes $D(\MHS)$ into a symmetric monoidal category. We wish to give an explicit formula for its tensor product. The dg-category $\ch{\MHS}^\Gamma$ is not a symmetric monoidal dg-category, so it does not directly induce a symmetric monoidal tensor on its homotopy category. To solve this issue, we consider a bigger enrichment of $\ch{\MHS}$ (cf. section $2$ of \cite{BeilRegSp}). Given $A \in \ch{\MHS}$, let
\[
 P_{\mathrm{DB}}(A,0) = \{ \omega \in \Lambda(t,dt) \otimes W_0 A \otimes \CC \, | \, \omega_0 \in W_0 A, \, \omega_1 \in F^0 W_0 A \otimes \CC \},
\]
and define, for $A,B \in \ch{\MHS}$,
\[ I \Gamma(A,B) := P_{\mathrm{DB}}(\underline{\Hom}(A,B),0). \]
The composition is induced by composition of maps and the commutative product of $\Lambda(t,dt)$. This structure makes $\ch{\MHS}$ a symmetric monoidal dg-category, which we denote by $\ch{\MHS}^{I\Gamma}$. By Lemma \ref{Delcomp-equiv-lemm} and Lemma \ref{ev-s-prod-lemm}, the complexes $\Gamma(A,B)$ and $I\Gamma(A,B)$ are homotopy equivalent and there is an equivalence of categories
\begin{equation}
  \label{gammaequiv-eq}
  \Ho_{dg} \big(\ch{\MHS}^\Gamma \big) \simeq \Ho_{dg} \big(\ch{\MHS}^{I\Gamma} \big).
\end{equation}
Moreover, the natural map
\[ D(\MHS) \to \Ho_{dg} \big(\ch{\MHS}^{I\Gamma} \big) \]
is a symmetric monoidal equivalence of categories. Under the equivalence (\ref{gammaequiv-eq}), the tensor product $\otimes^h$ in $\Ho_{dg}(\ch{\MHS}^{\Gamma})$ is given on morphisms by
\[ f \otimes^h g = (f_\Bbbk \otimes g_\Bbbk, f_\CC \otimes g_\CC, f_h \otimes g_\Bbbk + (-1)^{|g|} f_\CC \otimes g_h), \qquad \text{for } f,g \in \Gamma(A,B).\]

\subsection{Interpretation of the first obstruction}
Denote by $\mathrm{CAlg}(D(\MHS))$ the category of commutative algebras in $D(\MHS)$ for the derived tensor product. Given $H$ a graded commutative algebra, recall the DB-Harrison cohomology $\mathrm{Harr}_{\mathrm{DB}}^{*,*}(H)$ (Definition \ref{hodgeharr-coho-defi}). Recall also the first obstruction to formality of mixed Hodge diagrams of Theorem \ref{obst-theo}:
\[ \theta_3 \in \mathrm{Harr}_{\mathrm{DB}}^{3,-1}(H^*(A)), \quad \text{for } A \in \mathsf{MHD}. \]
\begin{theo}
  \label{phi2-theo}
  Let $A \in \mathsf{MHD}$. If the obstruction $\theta_3$ is trivial then $A$ is isomorphic to $H^*(A)$ in $\mathrm{CAlg}(D(\MHS))$. Moreover, if $A$ has pure cohomology, then the converse follows.
\end{theo}
\begin{proof}
  Any mixed Hodge diagram is quasi-isomorphic to a cdga in mixed Hodge structures (see Proposition \ref{Pmhalgmodels-prop}), so we assume that $A \in \mathsf{cdga}(\MHS)$. By Lemma $4.8$ of \cite{CG2}, $A \cong H^*(A)$ in $D(\MHS)$. Let 
  \[ f = [f_\Bbbk,f_\CC,f_h] \in H^0(\Gamma(A,H^*(A))) \]
  be an isomorphism. Then, $f$ is a map of algebras in $D(\MHS)$ if and only if the following diagram commutes
  \[ \begin{tikzcd}
        H^*(A) \otimes H^*(A) \arrow[r,"\mu"] \arrow[d,"f \otimes^h f"] & H^*(A) \arrow[d,"f"] \\
        A \otimes A \arrow[r,"\mu"] & A,
     \end{tikzcd} 
  \]
  where the maps $\mu$ denote the products of $A$ and $H^*(A)$. That is, if we have the equality
  \[ f \mu - \mu (f \otimes^h f) = 0 \in H^0(\Gamma(H^*(A)^{\otimes 2}, A)). \]
  The space $H^0(\Gamma(H^*(A)^{\otimes 2}, A))$ sits in a short exact sequence
  \[ 0 \to \mathrm{Ext}^1_{\MHS}(H^*(A)^{\otimes 2},H^*(A))^{-1} \xrightarrow{\iota} H^0(\Gamma(H^*(A)^{\otimes 2}, A)) \xrightarrow{\pi'} \Hom_{\MHS}^0(H^*(A)^{\otimes 2}, H^*(A)) \to 0. \]
  We may assume, without loss of generality, that $f_\Bbbk^* = f_\CC^* = id$. Then, 
  \[ \pi'(f \mu - \mu (f \otimes^h f)) = 0. \]
  Thus, there is an element
  \[ \gamma(f) \in \mathrm{Ext}^1_{\MHS}(H^*(A)^{\otimes 2},H^*(A))^{-1} \]
  such that 
  \[ \iota(\gamma(f)) = f \mu - \mu (f \otimes^h f). \]
  Hence, $f$ is a map of algebras if and only if $\gamma(f) = 0$. If 
  \[ g = [g_\Bbbk,g_\CC,g_h] \in H^0(\Gamma(A,H^*(A))) \]
  is another isomorphism such that $g_\Bbbk^* = g_\CC^* = id$, then 
  \[ \gamma(f) - \gamma(g) = \delta(\gamma(f-g)), \]
  where $\delta$ denotes the differential of $C_{\mathrm{DB}}(H^*(A))$ (see Definition \ref{hodgeharr-coho-defi}). Therefore, there is an isomorphism $A \cong H^*(A)$ in $\mathrm{Alg}(D(\MHS))$ if and only if $\gamma(f) = \delta(h)$ for some element 
  \[ h \in \mathrm{Ext}_{\MHS}^1(H^*(A),H^*(A))^{-1}. \]
  We now show that there is a model $B \in \mathsf{cdga}(\MHS)$ of $A \in \mathsf{MHD}$ which admits an isomorphism $f : H^*(A) \to B$ in $D(\MHS)$ such that $\gamma(f)$ is a representative of $\pi(\theta_3)$, where $\theta_3$ is the first obstruction of Theorem \ref{obst-theo} and $\pi$ is the map in (\ref{harrextmap-eq}). By Theorem \ref{htt-theo}, there is a $C_\infty$-mixed Hodge diagram 
  \[ H^*(A)_\infty = (H^*(A_\Bbbk),H^*(A_\CC),\hat{\varphi}) \] 
  quasi-isomorphic to $A \in \mathsf{MHD}$. Hence, $A$ is quasi-isomorphic to 
  \[ \Omega \tilde{\mathrm{B}} (H^*(A)_\infty) \in \mathsf{cdga}(\MHS) \] 
  (see Proposition \ref{Pmhalgmodels-prop}). Its Hodge filtration is given by $\Omega \tilde{\mathrm{B}} (\hat{\varphi})^{-1}(F)$, where $F$ is the Hodge filtration of $\Omega \tilde{\mathrm{B}} (H^*(A_\CC))$ (recall Remark \ref{Mhalgmodel-rema}). Let us denote a general element of $\Omega \tilde{\mathrm{B}}(H^*(A))$ by 
  \[ [x_1^1 | \dots | x^1_{k_1}] \otimes \cdots \otimes [x_1^m |\dots | x^m_{k_m}], \quad \text{for } x^i_j \in H^*(A), \]
  where $[x_1^i | \dots | x^i_{k_i}]$ denotes an element in $\tilde{\mathrm{B}}(H^*(A))$ and $\otimes$ denotes the tensor product of the cobar construction. The inclusion
  \[
    i : H^*(A) \to \Omega \tilde{\mathrm{B}}(H^*(A)_\infty) 
  \]
  is a quasi-isomorphism of complexes of mixed Hodge structures. Denote also by $i$ the isomorphism 
  \[ [i,i,0] \in H^0(\Gamma(H^*(A),\Omega \tilde{\mathrm{B}}(H^*(A)_\infty))). \]
  Then, we have that
  \[ i \mu - \mu(i \otimes^h i)(a,b) = (a \cdot b - a \otimes b, a \cdot b - a \otimes b, 0). \]
  Note that in $\Omega \tilde{\mathrm{B}} (H^*(A)_\infty)$,
  \[ d([a | b]) = d([a | b] - \hat{\varphi}_2(a,b)) = a \cdot b - a \otimes b \]
  and $[a | b] - \hat{\varphi}_2(a,b) \in F^{p+q}$, for $a \in F^p H^*(A)$ and $b \in F^q H^*(A)$. Hence, 
  \[ \gamma(i \mu - \mu(i \otimes^h i)) = -\hat{\varphi}_2 \]
  and the result follows from the fact that $[\hat{\varphi}_2] = \pi(\theta_3) \in \mathrm{Harr}_{\mathrm{Ext}}^2(H^*(A))$ (see Proposition \ref{obsrep-prop}). If $H^*(A)$ is pure, then the vanishing of $\varphi_2 = \pi(\theta_3)$ is equivalent to the vanishing of $\theta_3 \in \mathrm{Harr}_{\mathrm{DB}}^{3}(H^*(A))$ (see Theorem \ref{obst-coro}).
\end{proof}

Recall that $H^*_{\mathrm{DB}}(A)$ is a square zero extension of $\mathrm{Hdg}(A)$ (see  Definition \ref{sqzeroext-defi}). As a consequence of Theorem \ref{phi2-theo}, we have the following result.
\begin{coro}
  \label{phi2split-coro}
  Let $A \in \mathsf{MHD}$. If the first obstruction $\theta_3$ to formality of $A$ is trivial, then $H^*_\mathrm{DB}(A)$ is a split square zero extension of $\mathrm{Hdg}(A)$.
\end{coro}
\begin{proof}
  If $\theta_3 = 0$, then $A \cong H^*(A)$ in $\mathrm{CAlg}(D(\MHS))$. By Proposition \ref{delcompqiso-prop}, this implies that there is commuting diagram 
  \[ \begin{tikzcd}
        H_\mathrm{DB}^*(A) \arrow[r] \arrow[d,"\cong"] & \mathrm{Hdg}(A) \arrow[d,equal] \\
        H_\mathrm{DB}(H^*(A)) \arrow[r,"\pi"] & \mathrm{Hdg}(A) 
  \end{tikzcd}, \]
  of morphisms of algebras. The result follows from the fact that $\pi$ admits a right inverse
  \[ [a] \mapsto [[a],[a],0], \]
  which is a map of algebras.
\end{proof}
\section{Products in Deligne-Beilinson Cohomology}
\label{DBprods-sec}
We now define double and triple Deligne-Beilinson Massey products and relate these to mixed Hodge formality.
\subsection{Double Deligne-Beilinson Massey products}
For simplicity of notation, we treat the case of mixed Hodge diagrams $A$ of length $2$,
\[ \begin{tikzcd}[column sep = small]
    & (A_\CC',W) & \\
    (A_\Bbbk,W) \otimes \CC \arrow[ru,"\varphi_{1}"] & &\arrow[lu,"\varphi_2"'] (A_\CC,W).
  \end{tikzcd} \]
But all the following definitions and results can be directly extended to mixed Hodge diagrams of general length. Let $A \in \mathsf{MHD}_2$ and 
\[ a \in (2 \pi i)^{p_1} W_{2p_1} H^{2p_1 + n_1}(A_\Bbbk) \cap F^{p_1} W_{2p_1} H^{2p_1 + n_1}(A_\CC), \] 
\[ b \in (2 \pi i)^{p_2} W_{2p_1} H^{2p_2 + n_2}(A_\Bbbk) \cap F^{p_2} W_{2p_2} H^{2p_2 + n_2}(A_\CC) \hspace{1ex} \] 
be such that $ab = 0$. Here, we identify $H^*(A_\CC) = H^*(A_\CC') = H^*(A_\Bbbk) \otimes \CC$ using the comparison morphisms. Then there exist
\begin{enumerate}
  \item cycles $a_\Bbbk \in (2 \pi i)^{p_1} W_{2p_1} A_\Bbbk$ and $a_\CC \in F^{p_1} W_{2p_1} A_\CC$ and an element $a_h \in W_{2p_1} A_\CC'$ such that 
  \[ [a_\Bbbk] = [a_\CC] = a \quad \text{and} \quad \varphi_1(a_\Bbbk) - \varphi_2(a_\CC) = d a_h, \]
  \item \vspace{1ex} cycles $b_\Bbbk \in (2 \pi i)^{p_2} W_{2p_2} A_\Bbbk$ and $b_\CC \in F^{p_2} W_{2p_2} A_\CC$ and an element $b_h \in W_{2p_2} A_\CC'$ such that 
  \[ [b_\Bbbk] = [b_\CC] = b \quad \text{and} \quad \varphi_1(b_\Bbbk) - \varphi_2(b_\CC) = db_h, \] 
  \item \vspace{1ex} an element $x_\Bbbk \in (2 \pi i)^{p_1+p_2} W_{2(p_1+p_2)}A_\Bbbk$ such that $a_\Bbbk b_\Bbbk = d x_\Bbbk$ and \vspace{1ex}
  \item an element $x_\CC \in F^{p_1 + p_2} W_{2(p_1+p_2)}A_\CC$ such that  $a_\CC b_\CC = d x_\CC$.
\end{enumerate}
\vspace{1ex}
Note that $a_{\mathrm{DB}} = [a_\Bbbk,a_\CC,a_h]$ and $b_{\mathrm{DB}} = [b_\Bbbk,b_\CC,b_h]$ are classes in $H^*_\mathrm{DB}(A)$ which map to $a$ and $b$ through the projection
\[ H^*_\mathrm{DB}(A) \to \mathrm{Hdg}(A). \]
Since $ab = 0$, the product 
\[ a_{\mathrm{DB}} \cdot b_{\mathrm{DB}} \in H^{n_1+n_2}_\mathrm{DB}(A,p_1+p_2)\]
is in the kernel of this projection, so there is a class in $J^{p_1+p_2}H^{2(p_1+p_2) + n_1+n_2-1}(A_\Bbbk)$ which maps to it. This class is given by
\[ \Big[\frac{1}{2} \Big(a_h (\varphi_1(b_\Bbbk) + \varphi_2(b_\CC)) + (-1)^{|a|} (\varphi_1(a_\Bbbk) + \varphi_2(a_\CC)) b_h \Big) + \varphi_2(x_\CC) - \varphi_1(x_\Bbbk) \Big]. \]
This leads us to the following definition.

\begin{defi}
  The \textit{double Deligne-Beilinson Massey product} of $a$ and $b$ is the class
  \begin{equation}
    \label{double-DB-eq}
  \langle a,b \rangle_{\mathrm{DB}} \in \frac{J^{p_1 + p_2}H^{2(p_1 + p_2) + n_1 + n_2 -1 }(A_\Bbbk)}{a \cdot J^{p_2} H^{2p_2 + n_2-1}(A_\Bbbk) + J^{p_1} H^{2p_1 + n_1-1}(A_\Bbbk) \cdot b} 
  \end{equation}
  which maps to $a_{\mathrm{DB}} \cdot b_{\mathrm{DB}}$ under the inclusion $JH^*(A) \hookrightarrow H^*_{\mathrm{DB}}(A)$.
\end{defi}

\begin{prop}
  The double Deligne-Beilinson Massey product of $a$ and $b$ is well defined and a morphism $g : A \to B$ of mixed Hodge diagrams maps $\langle a,b \rangle_\mathrm{DB}$ to $\langle g_\CC^*(a),g_\CC^*(b) \rangle_\mathrm{DB}$. If $g$ is a quasi-isomorphim then $\langle a,b \rangle_\mathrm{DB}$ is trivial if and only if $\langle g_\CC^*(a),g_\CC^*(b) \rangle_\mathrm{DB}$ is trivial.
\end{prop}
\begin{proof}
  Given other choices 
  \begin{alignat*}{2}
    &dx_\Bbbk' = a_\Bbbk b_\Bbbk, \quad &&d x_\CC' = a_\CC b_\CC, \\
    &d a_h' = \varphi_1(a_\Bbbk) - \varphi_2(a_\CC), \quad &&b_h = \varphi_1(b_\Bbbk) - \varphi_2(b_\CC), 
  \end{alignat*}
  denote by $\langle a,b \rangle_\mathrm{DB}'$ the product obtained with these choises. Then, we have that
  \[ \langle a,b \rangle_\mathrm{DB} -  \langle a,b \rangle_\mathrm{DB}' = [x_\CC - x_\CC'] - [x_\Bbbk - x_\Bbbk'] + [a_h - a_h']b + (-1)^{|a|} a [b_h - b_h'].\]
  Hence, $\langle a,b \rangle_\mathrm{DB}$ and $\langle a,b \rangle_\mathrm{DB}'$ give the same class in the quotien of (\ref{double-DB-eq}). A similar argument shows that different choices of cycles $a_\Bbbk'$, $a_\CC'$, $b_\Bbbk'$ and $b_\CC'$ give also the same Deligne-Beilinson product. 
\end{proof} 
Note that if $H_\mathrm{DB}^*(A)$ is a split square zero extension, then every double Deligne-Beilinson Massey product is automatically trivial. Thus, by Corollary \ref{phi2split-coro}, double DB Massey products obstruct mixed Hodge formality. These products are more explicitly related to the first obstruction in the following way. Evaluating at $a \otimes b$ gives a well-defined morphism
\begin{align*}
  \mathrm{Harr}_{\mathrm{DB}}^2(H^*(A_\Bbbk)) &\xrightarrow{ev_{a \otimes b}} \frac{J^{p_1 + p_2}H^*(A_\Bbbk)}{a \cdot J^{p_2} H^*(A_\Bbbk) + J^{p_1} H^*(A_\Bbbk) \cdot b} \\
  [m_\Bbbk,m_\CC,\hat{\varphi}] &\mapsto \hat{\varphi}(a \otimes b)
\end{align*} 
It is well-defined because
\begin{align*}
   ev_{a \otimes b}(\delta (g_\Bbbk,g_\CC,h)) &= g_\Bbbk(a \otimes b) - g_\CC(a \otimes b) - \delta h(a \otimes b) \\
   &= g_\Bbbk(a \otimes b) - g_\CC(a \otimes b) - (-1)^{|a||h|} a \otimes h(b) - h(a) \otimes b,
\end{align*}
which is trivial in the quotient. Consider the first obstruction 
\[ \theta_3 \in \mathrm{Harr}_{\mathrm{DB}}^2(H^*(A_\Bbbk)) \] 
to formality of $A \in \mathsf{MHD}$. 
\begin{prop}
  We have that
  \[ \langle a, b \rangle_\mathrm{DB} = ev_{a \otimes b} (\theta_3). \]
\end{prop}
\begin{proof}
  By Proposition \ref{Pmhalgmodels-prop}, we can assume that $A$ is in $\mathsf{cdga}(\MHS)$ and is of the form
   \[ A = \Omega \tilde{\mathrm{B}} H^*(A)_\infty, \]
  where $H^*(A)_\infty$ is a $C_\infty$-mixed Hodge diagram with underlying complex $H^*(A_\Bbbk)$ and comparison morphism 
  \[ \hat{\varphi} : H^*(A_\Bbbk) \otimes \CC \to H^*(A_\Bbbk) \otimes \CC \]
  satisfying $\hat{\varphi}_1 = id$. The Hodge filtration of $A$ is then given by $\Omega \tilde{\mathrm{B}} \hat{\varphi}^{-1}(F)$, where $F$ is the free extension of the Hodge filtration of $H^*(A_\CC)$. Denote a general element of $\Omega \tilde{\mathrm{B}}(H^*(A))$ by 
  \[ [x_1^1 | \dots | x^1_{k_1}] \otimes \cdots \otimes [x_1^m |\dots | x^m_{k_m}], \quad \text{for } x^i_j \in H^*(A), \]
  where $[x_1^i | \dots | x^i_{k_i}]$ denotes an element in $\tilde{\mathrm{B}}(H^*(A))$ and $\otimes$ denotes the tensor product of the cobar construction. Let $D$ the differential of $\Omega \tilde{\mathrm{B}} H^*(A)_\infty$. Then, since $a \cdot b = 0$, we have that
  \[ D([a | b]) = - a \otimes b. \]
  Moreover, $\Omega \tilde{\mathrm{B}} \hat{\varphi}^{-1}([a|b]) = [a|b] - \hat{\varphi}_2(a \otimes b)$. Therefore, to compute $\langle a, b \rangle_\mathrm{DB}$ we can choose $a_\Bbbk = a_\CC = a$, $a_h = 0$, $b_\Bbbk = b_\CC = b$, $b_h = 0$, $x_\Bbbk = -[a | b]$ and $x_\CC = -[a | b] + \hat{\varphi}_2(a \otimes b)$. Moreover, by Proposition \ref{obsrep-prop}, $\hat{\varphi}_2$ is the third component of $\theta_3$. The product is then 
  \[ \langle a, b \rangle_\mathrm{DB} = [\hat{\varphi}_2(a \otimes b)] =  ev_{a \otimes b}(\theta_3). \]
\end{proof}
In certain situations, double Deligne-Beilinson Massey products are automatically trivial. For instance, if the cohomology of $A$ is concentrated in even degrees, then
\[ H^{2n}_\mathrm{DB}(A) \cong \mathrm{Hdg}^{2n}(A), \quad H^{2n-1}_\mathrm{DB}(A) \cong J H^{2n-1}(A). \]
Thus, given $a \in \mathrm{Hdg}^{2n_1}(A)$ and $b \in \mathrm{Hdg}^{2n_2}(A)$ such that $ab = 0$, then 
\[ \langle a,b \rangle_\mathrm{DB} \in J H^{2(n_1+n_2)}(A) = 0. \]
Another case is that of $\Bbbk = \RR$ and when $A \in \mathsf{MHD}_\RR$ has pure cohomology, which we treat in the following section. In both these situations, we can give a good definition of \textit{triple Deligne-Beilinson Massey products}. This is the content of the following section.

\subsection{Triple Deligne-Beilinson Massey products}
\label{3DB-massey-sec}
In this section, we define triple Deligne-Beilinson Massey products and relate them to mixed Hodge formality. As was mentioned in the introduction, these products have been studied before (see \cite{Den}, \cite{Wen}, \cite{MPDBcoho-thesis}, \cite{polysymb}, \cite{MPDifcohostack}). There are however slight differences between the products in the literature and the ones studied here. We define these triple products in the setting of mixed Hodge diagrams and study only products of Hodge classes, as these are the ones more naturally related to mixed Hodge formality. In the definition, we thus lift the elements in $\mathrm{Hdg}^*(A)$ (that is, the Hodge classes) to $\mathrm{H}_\mathrm{DB}^*(A)$ and so the products are well-defined when there is a canonical lift. As was stated in the previous section, this is the case when
\begin{enumerate}
    \item the cohomology $H^*(A)$ is concentrated in even degrees or
    \item $A$ is a real mixed Hodge diagram and $H^*(A)$ is $1$-pure.
\end{enumerate}
The following definition and results are applicable in both cases, but, for simplicity, we treat only the case of $\Bbbk = \RR$ and $A \in \mathsf{MHD}_\RR$ having pure cohomology
\[ H^n(A_\CC) = \bigoplus_{p+q = n} H^{p,q}. \]
In this case, we have the following identities:
\begin{align*}
   J^{p} H^{n}(A_\RR) = 0, \qquad &\text{ for } n \geq 2p-1, \\
   (2 \pi i)^p W_{2p} H^{n}(A_\RR) \cap F^p W_{2p} H^{n}(A_\CC) = 0, \qquad &\text{ for } n \neq 2p.
\end{align*}
These identities imply that
\begin{alignat}{2}
  \label{DBcoho1purereal-eq}
  &H_\mathrm{DB}^{n}(A,p) \cong J^pH^{2p + n-1}(A), &&\qquad \text{ for } n < 0, \\
  &H_\mathrm{DB}^{0}(A,p) \cong \mathrm{Hdg}^{2p}(A,p), &&~ \nonumber \\
  &H_\mathrm{DB}^{n}(A,p) = 0, &&\qquad \text{ for } n > 0. \nonumber
\end{alignat}
Note that the isomorphism $H_\mathrm{DB}^{0}(A,p) \cong \mathrm{Hdg}^{2p}(A,p)$ implies that there is canonical lift of Hodge classes and hence all double DB Massey products are trivial.
\begin{rema}
  \label{phi2=0-rema}
  By Proposition \ref{firstobswelldef-prop}, when $H^*(A)$ is concentrated in even degrees or when $\Bbbk = \RR$ and $H^*(A)$ is pure, the first obstruction $\phi_2$ is trivial, which is a stronger condition than the vanishing of double DB Massey products.
\end{rema}
Now let $a \in \mathrm{Hdg}^{2p_1}(A,p_1)$, $b \in \mathrm{Hdg}^{2p_2}(A,p_2)$ and $c \in \mathrm{Hdg}^{2p_3}(A,p_3)$ be such that 
\[ ab = 0 = bc. \] 
Let $a_\mathrm{DB} \in P_\mathrm{DB}(A)$ be a cycle such that 
\[ [a_\mathrm{DB}] = a \in H^{0}_\mathrm{DB}(A;p_1) \cong \mathrm{Hdg}^{2p_1}(A,p_1). \] 
Likewise, define $b_\mathrm{DB}$ and $c_\mathrm{DB}$ for $b$ and $c$. There exist $x, y \in P_\mathrm{DB}(A)$ such that 
\[ a_\mathrm{DB} b_\mathrm{DB} = dx, \qquad b_\mathrm{DB} c_\mathrm{DB} = dy. \]
Denote $p = p_1 + p_2 + p_3$. Given that $P_\mathrm{DB}(A)$ is a cdga, we can compute the triple Massey product of $[a_\mathrm{DB}], [b_\mathrm{DB}], [c_\mathrm{DB}]$, which is a class that lives in
\[ \frac{H^{-1}_\mathrm{DB}(A,p)}{a \cdot H^{*}_\mathrm{DB}(A,p) + H^{*}_\mathrm{DB}(A,p) \cdot c} \cong \frac{J^{p} H^{2p-2}(A)}{a \cdot J^{p_2+p_3}H^{*}(A) + J^{p_1+p_2} H^{*}(A) \cdot c}\]
\begin{defi}
  The \textit{triple Deligne-Beilinson Massey product} of $a,b$ and $c$ is the class
  \[\langle a, b, c \rangle_\mathrm{DB} = [x \cdot c_\mathrm{DB} - (-1)^{|a|} a_\mathrm{DB} \cdot y] \] 
\end{defi}
in the quotient
\[ \frac{J^{p} H^{2p - 2}(A)}{a \cdot J^{p_2+p_3}H^{*}(A) + J^{p_1+p_2} H^{*}(A) \cdot c}. \]
Note that a quasi-isomorphism of mixed Hodge diagrams induces an isomorphism in Deligne-Beilinson cohomology and preserves triple Deligne-Beilinson Massey products. There is also a relation between triple Deligne-Beilinson Massey products and obstructions to formality. First, note that there is a well-defined map
\begin{align}
  \label{phi3eval-eq}
  \mathrm{Harr}_{\mathrm{Ext}}^{3,-2}(H^*(A_\RR)) &\xrightarrow{ev_{a \otimes b \otimes c}} \frac{J^{p} H^{2p - 2}(A)}{a \cdot J^{p_2+p_3}H^{*}(A) + J^{p_1+p_2} H^{*}(A) \cdot c} \\
  \phi &\mapsto \phi(a \otimes b \otimes c). \nonumber
\end{align}
By Corollary \ref{firstobsnonvanish-coro}, the first obstruction to formality \[ \phi_2 \in \mathrm{Harr}_{\mathrm{Ext}}^{2,-1}(H^*(A_\RR)) \] of Theorem \ref{obst-coro} is trivial. Hence, the second obstruction 
\[ \phi_3 \in \mathrm{Harr}_{\mathrm{Ext}}^{3,-2}(H^*(A_\RR)) \] 
is defined. Moreover, the non-vanishing of $\phi_3$ implies non-formality of $A$.
\begin{prop}
  \label{3db-massey=obs-prop}
  We have the following equality:
  \[ \langle a, b, c \rangle_\mathrm{DB} = ev_{a \otimes b \otimes c}(\phi_3). \]
\end{prop}
\begin{proof}
  As observed in Remark \ref{phi2=0-rema}, the first obstruction $\phi_2$ is trivial. Hence, we can assume that $A$ is of the form 
  \[ A = \Omega \tilde{\mathrm{B}}(H^*(A)_\infty), \]
  where $H^*(A)_\infty = ((H^*(A_\RR),m_\RR),(H^*(A_\CC),m_\CC),\hat{\varphi}) \in C_\infty\textrm{-}\mathsf{MHD}$ satisfies 
  \[ (m_\RR)_3 = 0, \quad (m_\CC)_3 = 0, \quad \hat{\varphi}_1 = id, \quad \hat{\varphi}_2 = 0. \] 
  The Hodge filtration of $\Omega \tilde{\mathrm{B}}(H^*(A)_\infty)$ is given by $\Omega \tilde{\mathrm{B}}(\hat{\varphi})^{-1}(F)$, where $F$ is the Hodge filtration induced by the one of $H^*(A_\CC)$. Let
  \[ a_\mathrm{DB} = a \in P_\mathrm{DB}(\Omega \mathrm{B}(H_\infty))\]
  be the constant polynomial equal to $a$. Likewise, $b_\mathrm{DB} = b$ and $c_\mathrm{DB} = c$. Denote a general element of $\Omega \tilde{\mathrm{B}}(H^*(A))$ by 
  \[ [x_1^1 | \dots | x^1_{k_1}] \otimes \cdots \otimes [x_1^m |\dots | x^m_{k_m}], \quad \text{for } x^i_j \in H^*(A), \]
  where $[x_1^i | \dots | x^i_{k_i}]$ denotes an element in $\tilde{\mathrm{B}}(H^*(A))$ and $\otimes$ denotes the tensor product of the cobar construction. Note that
  \[ d([a | b]) = - a \otimes b \quad \text{and} \quad d([b | c]) = - b \otimes c. \]
  Since $\hat{\varphi}_2 = 0$, it follows that $[a | b] \in F^{p_1 + p_2}$ and $[b | c] \in F^{p_2 + p_3}$. Hence, to compute the triple Deligne-Beilinson Massey product, we can choose
  \[ x = - [a | b] \quad \text{and} \quad y = - [b | c] \quad \in P_\mathrm{DB}(\Omega \tilde{\mathrm{B}}(H^*(A)_\infty)) \]
  as constant polynomials. The class $\langle a, b, c \rangle_\mathrm{DB}$, seen as an element in $J^p H^*(A)$, is given by
  \[ [z_\CC-z_\RR], \]
  for choices of $z_\RR \in (2 \pi i)^p \Omega \tilde{\mathrm{B}}(H^*(A)_\infty)$, $z_\CC \in F^p\Omega \tilde{\mathrm{B}} (H^*(A)_\infty) \otimes \CC$ that satisfy
  \[ d z_\CC = - [ a | b] \otimes c + (-1)^{|a|} a \otimes [b | c]  = d z_\RR. \]
  Since $(m_\RR)_3 = 0$, we have that
  \[ d([a | b | c]) =  - (-1)^{|a|} a \otimes [b | c] + [a | b] \otimes c. \]
  Moreover, 
  \[ \Omega \tilde{\mathrm{B}} (\hat{\varphi})^{-1}([a | b | c]) = [a | b | c] - \hat{\varphi}_3(a \otimes b \otimes c) \]
  and, by Proposition \ref{obsrep-prop}, $\hat{\varphi}_3$ is a representative for the obstruction $\phi_3$. It then follows that
  \[ \langle a, b, c \rangle_\mathrm{DB} = [\hat{\varphi}_3(a \otimes b \otimes c)] = ev_{a \otimes b \otimes c}(\varphi_3).\]
\end{proof}
As a consequence, if there is some nontrivial triple Deligne-Beilinson Massey product, then $A$ is not mixed Hodge formal.

\begin{rema}
  \label{ABC=DB-rema}
  In section $5$ of \cite{MHF}, it is shown that for the mixed Hodge diagram $\Aa(X)$ associated to a compact Kähler manifold $X$, the element $\mathrm{ev}_{a,b,c}(\phi_3)$ gives the ABC-Massey product (see \cite{DaniToma}) of the Hodge classes $a$,$b$ and $c$. Proposition \ref{3db-massey=obs-prop} then implies that triple DB Massey products coincide with ABC-Massey products in this setting. This fact was communicated to us before by Jonas Stelzig and Dan Petersen, through a different proof. It implies that examples of compact Kähler manifolds with non-trivial triple ABC Massey products (see \cite{Steletall-nonformal}, \cite{StelMar-Mer}) yield examples of non-trivial triple DB Massey products. 
\end{rema}

\section{Deligne Beilinson products in Complex Geometry}
\label{cpxgeo-sec}
In this last section, we give applications to complex geometry. Namely, we compute double Deligne-Beilinson Massey products of algebraic cycles in terms of the Abel-Jacobi map and relate the existence of multiplicative Chow-Künneth decompositions to mixed Hodge formality.
\subsection{Mixed Hodge diagrams of smooth complex varieties}
By Navarro \cite{navarro}, there is a functor from complex algebraic varieties to mixed Hodge diagrams with coefficients in $\QQ$, localized at quasi-isomorphisms:
\[ \Aa : \mathsf{Var}_\CC \to \Ho(\MHD_\QQ) \]
Using this functor, we define mixed Hodge formality:
\begin{defi}
  A complex algebraic variety $X$ is \textit{mixed Hodge formal} if $\Aa(X)$ is a formal mixed Hodge diagram.
\end{defi}
As examples of mixed Hodge formal complex algebraic varieties, we have homogeneous compact Kähler manifolds, configuration spaces $F_k(\CC^n)$ of $k$ points in $\CC^n$, for $k < 2n$ and complex algebraic varieties with free cohomology, like complex tori (see \cite{MHF}).

In section $4$ of \cite{Beilinson}, Beilinson introduces the notion of absolute Hodge cohomology as the cohomology of the complex obtained by the composition of functors 
\[  \mathsf{Var}_\CC \xrightarrow{F} \Ho(\MHC_\QQ) \xrightarrow{C_{\mathrm{DB}}(-)} D(\QQ), \]
where $F$ denotes Deligne's functor to mixed Hodge complexes \cite{DeHII} and $C_{\mathrm{DB}}(-)$ is the functor (\ref{absHodgecpx-eq}). The functor $F$ is equivalent to $\Aa$ after forgetting the multiplicative structure and is constructed in the same way as $\Aa$ but using the usual derived functors $\RR j_*$ and $\RR \Gamma$ instead of the Thom-Whitney simple. Moreover, the functor $C_\mathrm{DB}(-)$ is equivalent to the functor $P_{\mathrm{DB}}(-)$ to cdga's (see Lemma \ref{Delcomp-equiv-lemm}). Therefore, the composition
\[ \mathsf{Var}_\CC \xrightarrow{\Aa} \Ho(\mathsf{MHD}_\QQ) \xrightarrow{P_\mathrm{DB}(-)} \Ho(\mathsf{cdga}_\QQ) \]
is equivalent to Beilinson's absolute Hodge complex, after forgetting the multiplicative structure. Recall the first obstruction to mixed Hodge formality of $X$ 
\[ \theta_3 \in \mathrm{Harr}_\mathrm{DB}^{3,-1}(H^*(X))\]
and that $H^*_{\mathrm{DB}}(X)$ is a square-zero extension
\[ 0 \to JH^*(\Aa(X)) \to H^*_{\mathrm{DB}}(X) \to \mathrm{Hdg}^*(\Aa(X)) \to 0. \]
Then, Corollary \ref{phi2split-coro} translates to the following proposition.
\begin{prop}
  If $\theta_3$ is trivial, then $H^*_{\mathrm{DB}}(X)$ is a split square-zero extension.
\end{prop}

\subsection{Double DB products of algebraic cycles}
In this section, we compute the double DB Massey products of cohomology classes which come from algebraic cycles of a smooth complex projective variety. Let $X$ be a smooth complex projective variety and recall that the Chow groups of $X$ are defined by 
\[ CH^i(X) = \Zz^i(X)/ \sim_{\mathrm{rat}}, \]
where $\Zz^i(X)$ denotes the free abelian group generated by algebraic cycles in $X$ of codimension $i$ and $\sim_{\mathrm{rat}}$ denotes rational equivalence. Two algebraic cycles $A,B$ are rationally equivalent if there exists an algebraic cycle $Z \in \Zz^i(X \times \PP^1)$ and two points $a,b \in \PP^1$ such that $Z$ intersects transversally $X \times a$ and $X \times b$ and 
\[ A = Z \cap (X \times a) \quad \text{and} \quad B = Z \cap (X \times b). \] 
The graded Chow group $CH^*(X)$ admits the structure of an algebra by means of the intersection product and there are cycle maps to Betti and Deligne-Beilinson cohomologies: 
\begin{alignat*}{3}
  c : CH^p(X) &\to H^{p,p}(X)_\QQ, \qquad c_\mathrm{DB} : CH^p(&&X) \to H_\mathrm{DB}^{0}(X;p) \\
  Z &\mapsto [Z], &&\; \; Z \mapsto [[Z],[Z],0],
\end{alignat*}
where $[Z]$ is the cohomology class of the topological cycle of $Z$. Here $H^{p,p}(X)_\QQ$ denotes $H^{p,p}(X) \cap H^{2p}(X;\QQ)$. These maps are morphisms of algebras and make the following diagam commute
\[ \begin{tikzcd}
    CH^p(X) \arrow[d,"c_\mathrm{DB}"] \arrow[rd,"c"] & \\
    H^{0}_\mathrm{DB}(X;p) \arrow[r] & H^{p,p}(X)_\QQ.
\end{tikzcd}
\]
Denote by $CH^*_{\mathrm{hom}}(X)$ the kernel of $c$, that is, the chow-group of homologically trivial cycles. It is an ideal of $CH^*(X)$. Since $X$ is smooth projective, $H^*(X)$ is pure so
\[ H^{p,p}(X)_\QQ = \mathrm{Hdg}^{2p}(X,p), \quad J^p H^{2p-1}(X;\QQ) = \frac{H^{2p-1}(X;\CC)}{H^{2p-1}(X;\QQ) + F^p H^{2p-1}(X;\CC)}. \]
There is thus a commuting diagram
\[ \begin{tikzcd}
  CH^p_{\mathrm{hom}}(X) \arrow[r] \arrow[d,"\mathcal{AJ}"] & CH^p(X) \arrow[d,"c_\mathrm{DB}"] \arrow[rd,"c"]  & \\
  0 \to J^p H^{2p-1}(X;\QQ) \arrow[r] & H^{0}_\mathrm{DB}(X;p) \arrow[r] & \mathrm{Hdg}^{2p}(X,p) \to 0.
\end{tikzcd}\] 
The map $\mathcal{AJ}$ is the Abel-Jacobi map, which is usually defined as the following morphism. First, note that Serre duality implies that 
\[ J^p H^{2p-1}(X;\CC) \cong \frac{(F^{m-p+1} H^{2m-2p+1}(X;\CC))^\vee}{H_{2m-2p+1}(X;\QQ)}, \]
where $m$ is the complex dimension of $X$. The Abel-Jacobi map is given by
\begin{align*}
  \mathcal{AJ} : CH^p_{\mathrm{hom}}(X) &\to \frac{(F^{m-p+1} H^{2m-2p+1}(X;\CC))^\vee}{H_{2m-2p+1}(X;\QQ)} \\
  Z &\mapsto \int_{\Gamma} \qquad \partial(\Gamma) = Z.
\end{align*}

\begin{prop}
  \label{AJdobDel-prop}
  Let $A \in CH^{p_1}(X)$ and $B \in CH^{p_2}(X)$ be such that $c(A) \cdot c(B) = 0$. Then, the double DB Massey product of $c(A)$ and $c(B)$ is given by
  \[ \langle c(A), c(B) \rangle_\mathrm{DB} = \mathcal{AJ}(A \cdot B). \] 
\end{prop}
\begin{proof}
  This is a simple consequence of the fact that $c$ is an algebra morphism and the commutativity of the previous diagram. Since 
  \[ c(A \cdot B) = c(A) \cdot c(B) = 0, \] 
  then $A \cdot B \in CH^{p_1 + p_2}_{\mathrm{hom}}(X)$. The double DB Massey product $ \langle c(A), c(B) \rangle_\mathrm{DB}$ is the element in
  \[J^{p_1 + p_2} H^{2(p_1 + p_2)-1}(X) \]
  that maps to $c_\mathrm{DB}(A) \cdot c_\mathrm{DB}(B) \in H^{0}_\mathrm{DB}(X;p_1+p_2)$. By the previous diagram, this is precisely $\mathcal{AJ}(A \cdot B)$.
\end{proof}
By Proposition \ref{AJdobDel-prop}, the smooth three-dimensional projective varieties considered in \cite{CCM} provide examples of non-trivial double DB Massey products. We write explicitly one of these examples.
\begin{exam}
  \label{nontrivDBprod-exam}
  Let $C \subset \CC \PP^3$ be a smooth non-rational curve and $P,Q \in C$, two distinct points. Consider $X = \mathrm{Bl}_{\tilde{C}} \mathrm{Bl}_{P,Q} \CC \PP^3$, where $\tilde{C}$ denotes the proper transform of $C$. Denote by $E_C$, $E_P$ and $E_Q$ the divisors of $C$, $P$ and $Q$ in $X$. These are cycles in $CH^1(X)$. Note that
  \[ [E_C] \cdot ([E_P] - [E_Q]) = 0 \in H^{2,2}(X)_\QQ \]
  By \cite{IntJac-ClGri}, $J^2 H^3(X;\CC) \cong J(C)$, where $J(C)$ is the jacobian of the curve $C$. Under this isomorphism, we have that
  \[ \mathcal{AJ}(E_C \cdot (E_P - E_Q)) = \int_P^Q \in J(C), \]
  which is non-trivial for most choices of $P$ and $Q$.
\end{exam}

\subsection{Multiplicative Chow-Künneth decomposition}
We show that the existence of a multiplicative Chow-Künneth decomposition of a smooth complex projective variety implies the vanishing of the first obstruction to mixed Hodge formality. Denote by $\Mm_{\mathrm{rat}}$ Grothendieck's category of pure Chow motives. This category is obtained in two steps starting with the category $Z_\sim \mathsf{SmProj}$ whose objects are (connected) smooth complex projective varieties and whose morphisms are degree $0$ correspondences. That is, given $X,Y$ smooth projective varieties,
\[ \Hom_{Z_\sim \mathsf{SmProj}}(X,Y) = \mathrm{Corr}^0(X,Y) = CH^{m}(X \times Y), \]
where $m$ is the complex dimension of $X$. Then $\Mm_{\mathrm{rat}}$ is obtained by taking the pseudo-abelian envolpe of $Z_\sim \mathsf{SmProj}_\QQ$ and inverting the motive of $\CC \PP^1$. The resulting category is symmetric monoidal and admits a symmetric monoidal functor from smooth projective varieties 
\begin{align*}
  \mathsf{SmProj}^{op} &\to \Mm_{\mathrm{rat}} \\
  X &\mapsto \mathfrak{h}(X).
\end{align*}
Any smooth projective variety $X$ is a cocomutative comonoid by the diagonal map $X \to X \times X$, so $\mathfrak{h}(X)$ is a commutative monoid in $\Mm_{rat}$. 
\begin{defi}
  A smooth projective variety $X$ of dimension $m$ is said to admit a \textit{Chow-Künneth decomposition} if its motive decomposes into a direct sum
  \[ \mathfrak{h}(X) = \mathfrak{h}^0(X) \oplus \cdots \oplus \mathfrak{h}^{2m}(X), \]
  such that for every $0 \leq i \leq 2m$, the Betti realization $H^*(\mathfrak{h}^i(X)) = H^i(X)$. Such a decomposition is said to be \textit{multiplicative} if the product $\mathfrak{h}(X) \otimes \mathfrak{h}(X) \to \mathfrak{h}(X)$, when restricted to $\mathfrak{h}^i(X) \otimes \mathfrak{h}^j(X)$ factors through $\mathfrak{h}^{i+j}(X)$.
\end{defi}
Deligne's functor to mixed Hodge complexes $\mathsf{SmProj}^{op} \to \Ho(\MHC_\QQ)$ (see section $3$ of \cite{DeHII}) factors through the category of pure Chow motives:
\[ \begin{tikzcd}
     \mathsf{SmProj}^{op} \arrow[rr] \arrow[rd] & & \Mm_{\mathrm{rat}} \arrow[dl]  \\
    & \Ho(\MHC_\QQ) & 
\end{tikzcd}\]
(see, for instance, chapter $5$ of \cite{MMLev}). Moreover, all functors in this diagram are symmetric monoidal. 
\begin{prop}
  If $X \in \mathsf{SmProj}$ admits a multiplicative Chow-Künneth decomposition, then the first obstruction to mixed Hodge formality $\phi_2$ vanishes.
\end{prop}
\begin{proof}
  Deligne's functor is equivalent to $\Aa : \mathsf{SmProj} \to \Ho(\mathsf{MHD}_\QQ)$ after forgetting to mixed Hodge complexes. If $X$ admits a multiplicative Chow-Künneth decomposition, then the image of $\mathfrak{h}(X)$ in $\Ho(\MHC_\QQ)$ is equivalent to $H^*(X)$ as objects in $\mathsf{CAlg}(D(\MHS))$. Thus, by the commutativity of the previous diagram, the same follows for $\Aa(X)$. The result then follows from Theorem \ref{phi2-theo}.
\end{proof}

\bibliographystyle{alpha}
\bibliography{bibliography}

\end{document}